\documentclass[12pt]{article}

\usepackage[english]{babel}
\usepackage[utf8]{inputenc}
\usepackage[top=1.5cm, bottom=2.5cm, left=2.5cm, right=2.5cm]{geometry}
\usepackage{mathrsfs}
\usepackage{amsmath}
\usepackage{amsthm}
\usepackage{amsfonts}
\usepackage{graphicx}
\usepackage{bbm}
\usepackage{hyperref}

\title{\textbf{LIMITS OF THE BOUNDARY OF RANDOM PLANAR MAPS}}

\author{Loïc Richier%
\thanks{CMAP, \'Ecole polytechnique, Palaiseau, France. Email: \texttt{loic.richier@polytechnique.edu}.\newline This work was partially accomplished at UMPA, \'Ecole Normale Supérieure de Lyon and supported by the grant ANR-14-CE25-0014 (ANR GRAAL).}}

\date{\today}

\newcommand{\m}{\mathbf{m}}
\newcommand{\ms}{\widehat{\m}}
\newcommand{\Vm}{\mathrm{V}(\m)}
\newcommand{\Em}{\mathrm{E}(\m)}
\newcommand{\Fm}{\mathrm{F}(\m)}
\newcommand{\Bm}{\partial\m}

\newcommand{\M}{\mathcal{M}}
\newcommand{\Ms}{\widehat{\M}}
\newcommand{\Mh}{\widehat{M}}
\newcommand{\Mp}{\M^{\bullet}}
\newcommand{\Z}{\mathbb{Z}}
\newcommand{\q}{\mathsf{q}}
\newcommand{\N}{\mathbb{N}}
\newcommand{\R}{\mathbb{R}}
\newcommand{\wq}{w_{\q}}
\newcommand{\dm}{\mathrm{d}_{\m}}
\newcommand{\Zq}{Z_{\q}}
\newcommand{\Zqu}{Z_{\q(u)}}
\newcommand{\Pq}{\mathbb{P}^{\bullet}_{\q}}
\newcommand{\Pqk}{\mathbb{P}^{(k)}_{\q}}
\newcommand{\Pqks}{\widehat{\mathbb{P}}^{(k)}_{\q}}
\newcommand{\Pqr}{\mathbb{P}_{\q,\rq}}
\newcommand{\Eq}{\mathbb{E}^\bullet_{\q}}
\newcommand{\fq}{f_{\q}}

\newcommand{\tr}{\mathbf{t}}
\newcommand{\Tr}{\mathbf{T}}
\newcommand{\Trs}{\mathscr{T}}
\newcommand{\Tinf}{\Tr_\infty}
\newcommand{\T}{\mathcal{T}}
\newcommand{\Tloc}{\T_\textup{loc}}
\newcommand{\tb}{\tr_{\bullet}}
\newcommand{\tw}{\tr_{\circ}}
\newcommand{\BDG}{\Phi_{\textup{BDG}}}
\newcommand{\JS}{\Phi_{\textup{JS}}}
\newcommand{\TC}{\Phi_{\textup{TC}}}
\newcommand{\GWn}{\mathsf{GW}_{\nu}}
\newcommand{\nub}{\nu_{\bullet}}
\newcommand{\nuw}{\nu_{\circ}}

\newcommand{\mnu}{m_{\nu}}
\newcommand{\snu}{\sigma^2_{\nu}}
\newcommand{\snub}{\sigma^2_{\nub}}
\newcommand{\GWnn}{\mathsf{GW}_{\nuw,\nub}}
\newcommand{\GWr}{\mathsf{GW}_{\rho}}
\newcommand{\rhob}{\rho_{\bullet}}
\newcommand{\rhow}{\rho_{\circ}}
\newcommand{\mrhob}{m_{\rho_{\bullet}}}
\newcommand{\mrhow}{m_{\rho_{\circ}}}
\newcommand{\mrho}{m_{\rho}}

\newcommand{\GWrr}{\mathsf{GW}_{\rhow,\rhob}}

\newcommand{\mub}{\mu_{\bullet}}

\newcommand{\mmu}{m_{\mu}}
\newcommand{\smu}{\sigma^2_{\mu}}

\newcommand{\On}{O(n)}

\newcommand{\Fs}{\widehat{F}}

\newcommand{\rsq}{\widehat{r_{\q}}}
\newcommand{\ls}{\widehat{\ell}}
\newcommand{\lb}{\bar{\ell}}

\newcommand{\cs}{\widehat{c}}
\newcommand{\lt}{\mathbf{l}}
\newcommand{\LT}{\mathcal{L}}
\newcommand{\Loop}{\textup{Loop}}
\newcommand{\Tree}{\textup{Tree}}
\newcommand{\Treeb}{\textup{\textbf{Tree}}}
\newcommand{\Loopb}{\textup{\textbf{Loop}}}
\newcommand{\Loopbb}{\overline{\Loopb}}
\newcommand{\Scoop}{\textup{Scoop}}
\newcommand{\dgh}{d_{\textup{GH}}}
\newcommand{\Lt}{\mathbf{L}}
\newcommand{\Lts}{\mathscr{L}}
\newcommand{\Linf}{\Lt_\infty}
\newcommand{\e}{\mathbbm{e}}

\newcommand{\dloc}{d_{\textup{loc}}}
\newcommand{\B}{\mathbf{B}}
\newcommand{\Bl}{\mathbf{B}^\leftarrow}
\newcommand{\Blr}{\mathbf{B}^\leftrightarrow}
\newcommand{\uh}{\widehat{u}}
\newcommand{\Twb}{T^{\circ,\bullet}}
\newcommand{\Tinfwb}{\Tinf^{\circ,\bullet}}
\newcommand{\Minf}{\mathbf{M}_\infty}
\newcommand{\Core}{\textup{Core}}
\newcommand{\IBHPM}{\textup{\textsf{IBHPM}}}
\newcommand{\Cut}{\textup{Cut}}
\newcommand{\Wm}{W_{-1}}
\newcommand{\qd}{\mathbf{q}}
\newcommand{\bell}{\boldsymbol{\ell}}

\newcommand{\Pnghk}{\mathbf{P}^{(k)}_{(n;g,h)}}

\renewcommand{\rq}{r_{\q}}

\newtheorem{Th}{\bf Theorem}[section]

\newtheorem{Prop}[Th]{\bf Proposition}

\newtheorem{Lem}[Th]{\bf Lemma}
\newtheorem{Cor}[Th]{\bf Corollary}

\theoremstyle{definition}

\newtheorem{Def}[Th]{\bf Definition}

\newtheorem{Rk}[Th]{\bf Remark}

\newtheorem*{Ack}{\bf Acknowledgements}
\newtheorem*{Not}{\bf Notation}

\begin{document}

\maketitle

\begin{abstract}
	We discuss asymptotics for the boundary of critical Boltzmann planar maps under the assumption that the distribution of the degree of a typical face is in the domain of attraction of a stable distribution with parameter $\alpha \in (1,2)$. First, in the dense phase corresponding to $\alpha\in(1,3/2)$, we prove that the scaling limit of the boundary is the random stable looptree with parameter $1/(\alpha-1/2)$. Second, we show the existence of a phase transition through local limits of the boundary: in the dense phase, the boundary is tree-like, while in the dilute phase corresponding to $\alpha\in(3/2,2)$, it has a component homeomorphic to the half-plane. As an application, we identify the limits of loops conditioned to be large in the rigid $\On$ loop model on quadrangulations, proving thereby a conjecture of Curien \& Kortchemski. 
\end{abstract}

\section{Introduction}\label{sec:Introduction}

The purpose of this work is to investigate local limits, in the sense of Angel \& Schramm, and scaling limits, in the Gromov-Hausdorff sense, of the boundary of bipartite Boltzmann planar maps conditioned to have a large perimeter. 

\medskip

\noindent\textbf{Model and motivation.} Given a sequence $\q=(q_1,q_2,\ldots)$ of nonnegative real numbers and a planar map $\m$ which is bipartite (i.e., with faces of even degree), the associated Boltzmann weight is
\[\wq(\m):=\prod_{f \in \textup{Faces}(\m)}q_{\deg(f)/2}.\] The sequence $\q$ is admissible if these weights form a finite measure on the set of rooted bipartite maps (with a distinguished oriented edge) that we call the Boltzmann measure with weight $\q$. We also say that $\q$ is critical if moreover the expected squared number of vertices of the map is infinite under this measure.

The scaling limits of Boltzmann maps conditioned to have a large number of faces (or vertices) have attracted a lot of attention. The first model to be considered was the uniform measure on $2p$-angulations, in which all faces have the same degree $2p$. In this case, Le~Gall~\cite{le_gall_topological_2007} proved the subsequential convergence towards a random metric space called the \textit{Brownian map}, first introduced by Marckert \& Mokkadem in~\cite{marckert_limit_2006} and whose distribution has been characterized later by Le Gall \cite{le_gall_uniqueness_2013} and Miermont \cite{miermont_brownian_2013}. This result has been extended by Le Gall \cite{le_gall_uniqueness_2013} to critical sequences $\q$ such that the degree of a typical face has exponential moments (while the first results on this model were obtained by Marckert \& Miermont \cite{marckert_invariance_2007}). The result also holds for critical sequences $\q$ such that the degree of a typical face has a finite variance, as shown in the recent work \cite{marzouk_scaling_2016} (such a sequence is called \textit{generic critical}, see Section \ref{sec:BoltzmannDistribution} for precise definitions). Convergence towards the Brownian map has also been established in the non-bipartite case in~\cite{miermont_invariance_2006,miermont_radius_2008}. All these results demonstrate the universality of the Brownian map, whose geometry is now well understood \cite{le_gall_scaling_2008,le_gall_geodesics_2010}. 

For a different behaviour to arise, Le Gall \& Miermont suggested in \cite{le_gall_scaling_2011} to assume, besides criticality, that the degree of a typical face is in the domain of attraction of a stable law with parameter $\alpha\in(1,2)$. The weight sequence $\q$ is then called \textit{non-generic critical} with parameter $\alpha$. Under slightly stronger assumptions, they proved the (subsequential) convergence towards a one-parameter family of random metric spaces called the \textit{stable maps} with parameter $\alpha$. These are supposed to be very different from the Brownian map because of large faces that remain present in the scaling limit. Their duals have been recently studied in~\cite{budd_geometry_2017,bertoin_martingales_2016}, but their geometry remains widely unknown. The stable maps are believed to undergo a phase transition at $\alpha=3/2$. In the regime $\alpha\in(1,3/2)$, called the \textit{dense phase}, the large faces of the map are supposed to be self-intersecting in the limit, while in the regime $\alpha\in(3/2,2)$, called the \textit{dilute phase}, they are supposed to be self-avoiding. The aim of this work is twofold: first, we identify the branching structure of the large faces in the dense phase via \textit{scaling limits}. Then, we establish the phase transition through \textit{local limits} of large faces.

\medskip

\noindent\textbf{Main results.} This paper adresses maps with a boundary, meaning that the face on the right of the root edge (the root face) is interpreted as the boundary $\Bm$ of the map $\m$. Precisely, we consider a Boltzmann map with weight $\q$ conditioned to have perimeter $2k$, say $M_k$, whose law is denoted by $\Pqk$. One can then interpret $\partial M_k$ as a typical face of degree $2k$ of a Boltzmann map. Our main result studies the scaling limit of $\partial M_k$, equipped with its graph distance.

\begin{Th}\label{th:ScalingDense}  Let $\q$ be a non-generic critical sequence with parameter $\alpha\in(1,3/2)$. For every $k\geq 0$, let $M_k$ be a map with law $\Pqk$. Then, there exists a slowly varying function $\Lambda$ such that in the Gromov-Hausdorff sense,
	\[\frac{\Lambda(k)}{ (2k)^{\alpha-1/2}}\cdot \partial M_k \underset{k\rightarrow \infty}{\overset{(d)}{\longrightarrow}} \Lts_{\beta}, \qquad \text{where} \qquad \beta:=\frac{1}{\alpha-\frac{1}{2}}\in(1,2)\]  and $\Lts_{\beta}$ is the random stable looptree with parameter $\beta$.
\end{Th} The stable looptrees $(\Lts_\beta : \beta\in (1,2))$ were introduced by Curien \& Kortchemski in \cite{curien_random_2014}, and can be seen as the stable trees of Duquesne \& Le Gall \cite{duquesne_random_2002,duquesne_probabilistic_2005} in which branching points are turned into topological circles. Stable looptrees also appear as the scaling limits of discrete \textit{looptrees} \cite{curien_random_2014}, which are informally collections of cycles glued along a tree structure. They have Hausdorff dimension $\beta$ a.s. \cite[Theorem 1.1]{curien_random_2014}.

The result of Theorem \ref{th:ScalingDense} covers the dense case. We believe that in the dilute and generic critical phases, the scaling limit of $\partial M_k$ is a circle. Furthermore, in the subcritical phase, the \textit{Continuum Random Tree} \cite{aldous_continuum_1991,aldous_continuum_1993} is expected to arise as a scaling limit. We will discuss this in Section \ref{sec:ScalingLimitsSection}.

\smallskip 

The local limits of Boltzmann maps with a boundary have been studied by Curien in \cite[Theorem 7]{curien_peeling_2016}. He proved that for any admissible weight sequence $\q$, we have in the local sense
\begin{equation}\label{eqn:LocalLimitBoltzmannMap}
	M_k \overset{(d)}{\underset{k\rightarrow \infty}{\longrightarrow}} \Minf.
\end{equation} The map $\Minf=\Minf(\q)$ is known as the Infinite Boltzmann Half-Planar Map with weight~$\q$ ($\q$-$\IBHPM$ for short). The infinite boundary $\partial \Minf$ of $\Minf$ is a.s.\ non-simple and has self-intersections, the cut vertices (or pinch points). Then, $\Minf$ can be decomposed into \textit{irreducible components}, that are maps with a simple boundary attached by cut vertices of $\partial \Minf$. When $\Minf$ has a unique infinite irreducible component, it is called the \textit{core}. For technical reasons, we rather deal with the \textit{scooped-out} map $\Scoop(\Minf)$, obtained by duplicating the edges of $\partial\Minf$ whose both sides belong to the root face. 

Naturally, $\Scoop(\Minf)$ is a local limit version of looptrees that we briefly describe (see Section \ref{sec:RandomInfiniteLooptrees} for details). Given a pair of offspring distributions $(\rhow,\rhob)$, an alternated two-type Galton-Watson tree is a random tree in which vertices at even (resp.\ odd) height have offspring distribution $\rhow$ (resp.\ $\rhob$) all independently of each other. As in the monotype case, we can make sense of such trees conditioned to survive, and denote the limiting infinite tree by $\Tinfwb=\Tinfwb(\rhow,\rhob)$. When $(\rhow,\rhob)$ is critical (meaning that the product of their means equals one), Stephenson established in \cite{stephenson_local_2016} that $\Tinfwb$ is a two-type version of \textit{Kesten's tree}, that is a.s.\ locally finite with a unique \textit{spine} (see \cite{kesten_subdiffusive_1986,lyons_probability_2016,abraham_local_2014} for details). Under suitable assumptions, we will prove in Proposition \ref{prop:ConvergenceJSTreeTwoType} that when $(\rhow,\rhob)$ is subcritical, $\Tinfwb$ has a.s.\ a unique vertex of infinite degree (at odd height). This phenomenon, known as \textit{condensation}, was first observed by Jonsson \& Stef\'ansson in \cite{jonsson_condensation_2010} (see also \cite{janson_simply_2012,abraham_local_2014-1,kortchemski_limit_2015}). We then define an infinite map $\Linf=\Linf(\rhow,\rhob)$ by taking each vertex at odd height in $\Tinfwb$ and connecting its neighbours by edges in cyclic order. Therefore, $\Linf$ has only finite faces in the critical regime, and a unique infinite face in the subcritical regime. Note that $\rhob$ dictates the size of the finite faces of $\Linf$. We can now state our local limit result.

\begin{Th}\label{th:LocalLimits}
	Let $\q$ be either subcritical, generic critical or non-generic critical with parameter $\alpha\in(1,2)$. For every $k\geq 0$, let $M_k$ be a map with distribution $\Pqk$ and let $\Minf$ be the $\q$-$\IBHPM$. Then, there exists probability measures $\nuw$ (geometric) and $\nub$ such that in the local sense
\[\Scoop(M_k) \overset{(d)}{\underset{k\rightarrow \infty}{\longrightarrow}}\Scoop(\Minf) \overset{(d)}{=} \Linf(\nuw,\nub).\] A phase transition is observed:
\begin{itemize}
	\item If $\q$ is subcritical or non-generic critical with parameter $\alpha\in(1,3/2)$, $(\nuw,\nub)$ is critical and $\Minf$ has only finite irreducible components.
	\item If $\q$ is non-generic critical with parameter $\alpha\in(3/2,2)$ or generic critical, $(\nuw,\nub)$ is subcritical and $\Minf$ has a well-defined core with an infinite simple boundary.
\end{itemize} Moreover, $\nub$ has finite variance if and only if $\q$ is subcritical. Otherwise, $\nub$ is in the domain of attraction of a stable distribution, with parameter $1/(\alpha-1/2)$ $($if $\alpha\in(1,3/2)$$)$, $\alpha-1/2$ $($if $\alpha\in(3/2,2)$$)$ or $3/2$ $($if $\q$ is generic critical$)$.
\end{Th}

In the dense phase, $\Minf$ is tree-like, while in the dilute phase, it has an irreducible component homeomorphic to the half-plane on which finite maps are grafted (see Figure \ref{fig:PhaseTransition}). In the subcritical and dense phases, the $\q$-$\IBHPM$ can even be recovered from $\Linf$ and a collection of independent maps with a simple boundary, as shown in Proposition \ref{prop:LocalLimitSubcriticalDense}. Such collections of random combinatorial structures attached to a tree also appear in \cite{stufler_limits_2016}. In the dilute and generic critical regimes, we expect the core of $\Minf$ to be the local limit of Boltzmann maps constrained to have a simple boundary when the perimeter goes to infinity (see Section \ref{sec:LocalLimitsSubsection} for more on this). The critical parameter $\alpha=3/2$ plays a special role that we discuss in Section \ref{sec:TheA2Case}.

\begin{figure}[ht]
	\centering
	\includegraphics[scale=.74]{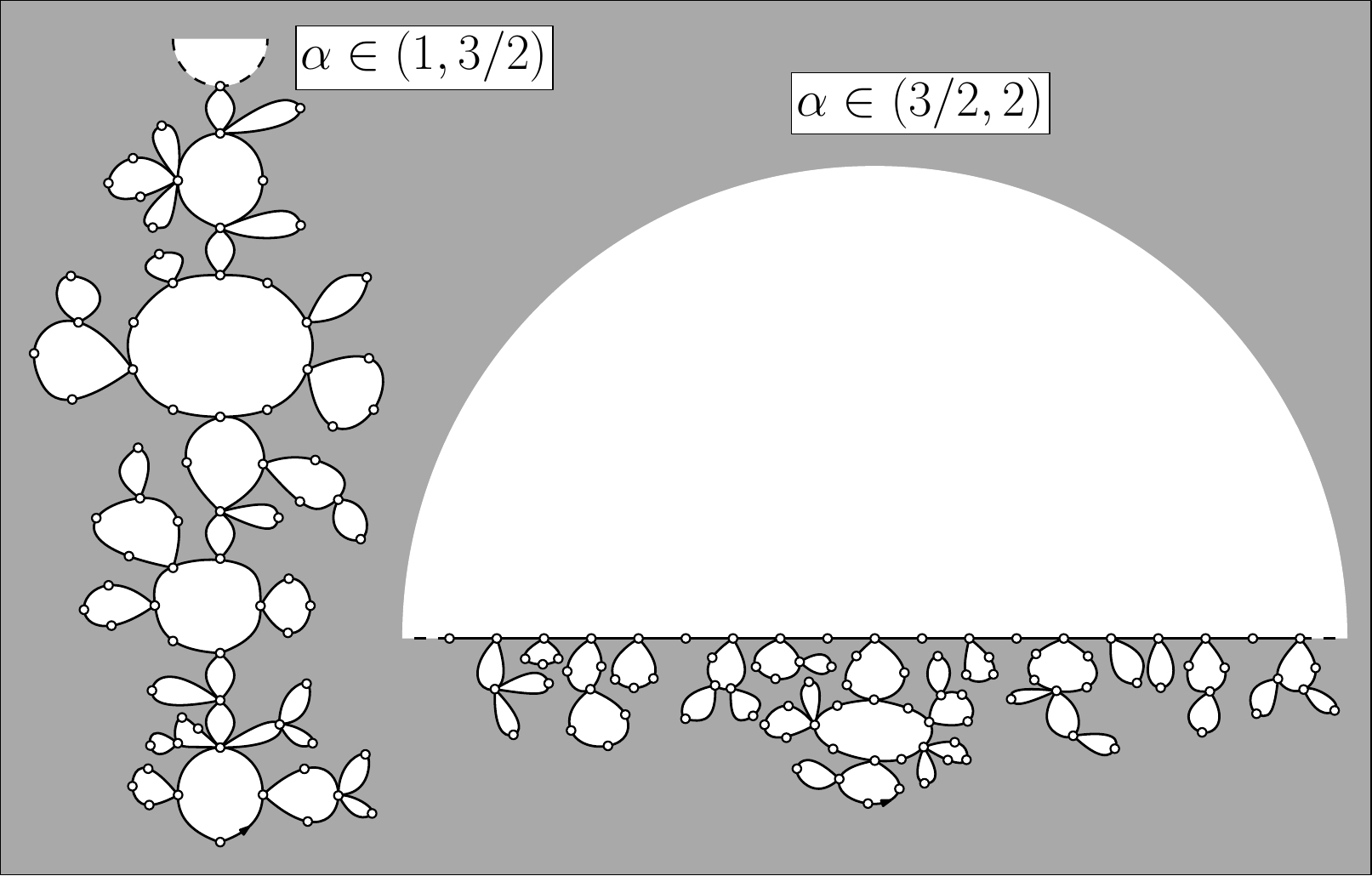}
	\caption{Schematic representation of $\partial \Minf$ for $\q$ non-generic critical with $\alpha\in(1,2)$.}
	\label{fig:PhaseTransition}
	\end{figure}

\medskip

\noindent\textbf{Applications to the rigid $\On$ loop model on quadrangulations.} The study of Boltzmann distributions such that $\q$ is non-generic critical with parameter $\alpha\in(1,2)$ is also motivated by the connection with statistical physics models on random maps. Here, we are interested in the \textit{rigid $\On$ loop model} on quadrangulations, studied by Borot, Bouttier \& Guitter in \cite{borot_recursive_2012}, see also \cite{chen_perimeter_2017,budd_peeling_2017}.

A \textit{loop-decorated} quadrangulation with a boundary $(\qd,\bell)$ is a planar map $\qd$ whose faces all are quadrangles (except the root face), together with a collection of disjoint closed simple paths $\bell=(\ell_1,\ell_2,\ldots)$ drawn on the dual of $\qd$, called \textit{loops} (which do not visit the root face). The \textit{loop configuration} $\bell$ is \textit{rigid} if all loops cross quadrangles through their opposite sides. Given $n\in (0,2)$ and $g,h\geq 0$, we define a measure on loop-decorated quadrangulations by \[W_{(n;g,h)}((\qd,\bell)):=g^{\#\mathrm{Faces}(\qd)-\vert \bell \vert}h^{\vert \bell \vert}n^{\#\bell},\] where $\vert \bell\vert$ is the total length of the loops and $\#\bell$ the number of loops. We say that the triplet $(n;g,h)$ is admissible if for every $k\geq 0$ this induces a probability measure $\Pnghk$ on loop-decorated quadrangulations with perimeter $2k$ (see Figure \ref{fig:OnWeights} for an illustration). The case $k=1$ corresponds to the rigid $\On$ model on quadrangulations of the sphere, by gluing the two edges of the boundary together.

\begin{figure}[ht]
	\centering
	\includegraphics[scale=.74]{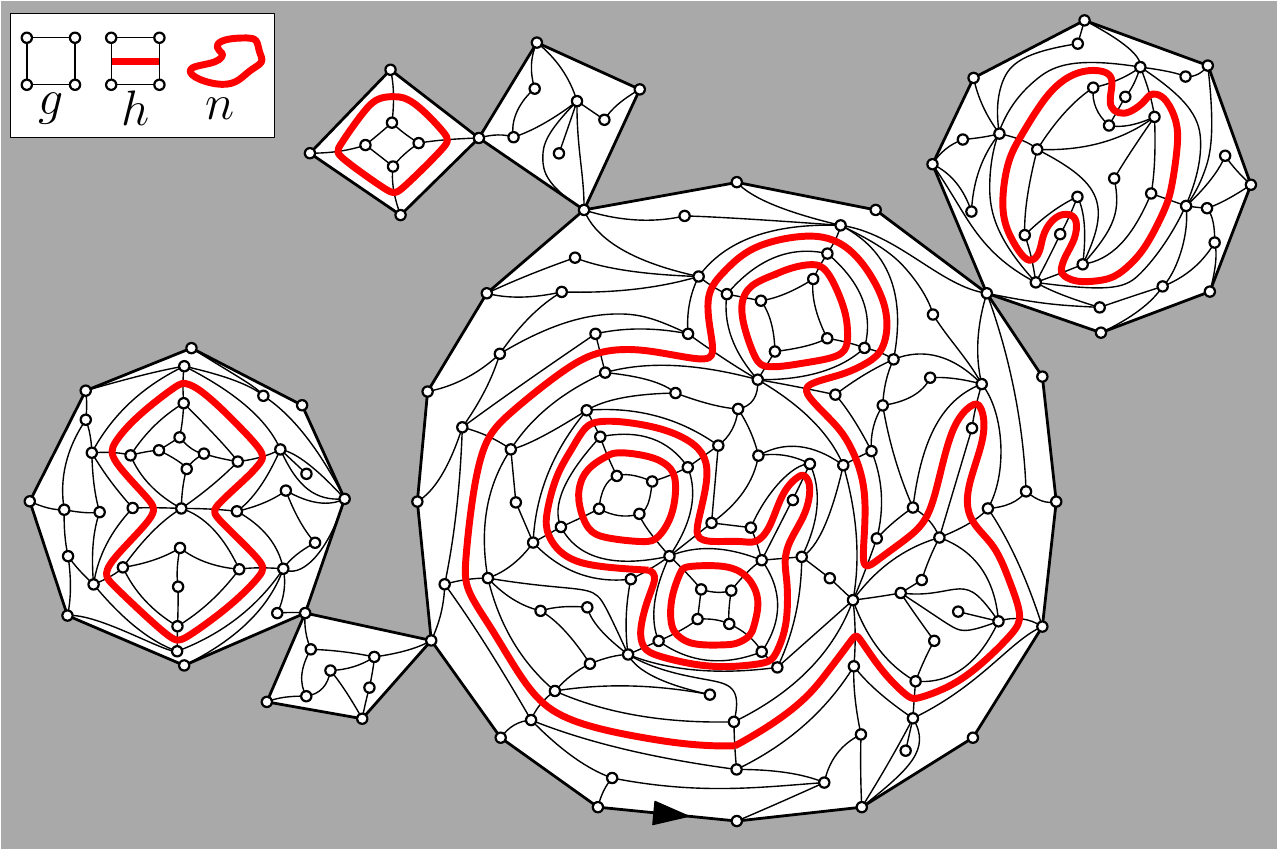}
	\caption{A rigid loop configuration $\bell$ on a quadrangulation with a boundary $\qd$.}
	\label{fig:OnWeights}
	\end{figure}

In \cite{borot_recursive_2012}, Borot, Bouttier \& Guitter introduced the \textit{gasket} of a loop-decorated quadrangulation, obtained by pruning the interior of the outermost loops (with respect to the root). They proved that under $\Pnghk$, the gasket is a Boltzmann map with law $\Pqk$, where $\q=\q(n;g,h)$ is the solution of \cite[Equation 2.3]{borot_recursive_2012}. This leads to a classification of the parameters $(n;g,h)$ in regimes depending of the type of the sequence $\q$. It has been argued in \cite{borot_recursive_2012} (and fully justified in \cite[Appendix]{budd_peeling_2017}) that the model admits a complete phase diagram shown in \cite[Figure 12]{borot_recursive_2012}. For every $n\in(0,2)$, there exists a critical line $h=h_c(n;g)$ that separates subcritical and ill-defined parameters. The regime changes along the critical line. There is a special point $(g^*(n),h^*(n))$ such that the parameters are non-generic critical with parameter $\alpha=3/2-\arccos(n/2)/\pi$ (dense) for $g<g^*$, and generic critical for $g>g^*$. The special point $(g^*,h^*)$ itself is non-generic critical with parameter $\alpha=3/2+\arccos(n/2)/\pi$ (dilute).

In this work, we are motivated by the study of the geometry of large loops in the rigid $\On$ model on quadrangulations. More generally, the interfaces in statistical physics models on maps are of great interest. In \cite{curien_percolation_2014}, Curien and Kortchemski studied percolation on uniform triangulations of the sphere. They proved that the boundary of a critical percolation cluster conditioned to be large admits as a scaling limit the random stable looptree with parameter $3/2$. They also conjectured that the whole family $(\Lts_\beta : \beta \in (1,2))$ appears as scaling limit of large loops in the $\On$ model on triangulations. The following application of Theorem \ref{th:ScalingDense} proves this conjecture for the rigid $\On$ model on quadrangulations.

\begin{Cor}\label{cor:ScalingLoops} Let $n\in(0,2)$, $g\in[0,g^*(n))$ and $h:=h_c(n;g)$. For every $k\geq 0$, let $(Q_k,L_k)$ be a loop-decorated quadrangulation with law $\Pnghk$. Then, there exists a constant $C=C(n,g,h)$ such that in the Gromov-Hausdorff sense,
	\[\frac{C}{ (2k)^{1/\beta}}\cdot \partial Q_k \underset{k\rightarrow \infty}{\overset{(d)}{\longrightarrow}} \Lts_{\beta}, \qquad \text{where} \quad  \beta:= \left. 1 \middle/ \left(1-\frac{1}{\pi} \arccos\left(\frac{n}{2}\right)\right)\right. \in(1,2).\]
\end{Cor}

Note that the value of $\beta$ fits the prediction of \cite{curien_percolation_2014}. We also obtain the local limits of large loops in the $\On$ model from Theorem \ref{th:LocalLimits}. These results are obtained by applying Theorems \ref{th:ScalingDense} and \ref{th:LocalLimits} to the gasket of the loop-decorated quadrangulation $(Q_k,L_k)$ (see also Remark \ref{rk:ConstantOn}). At first glance, they hold only for the boundary of loop-decorated quadrangulations. However, by the gasket decomposition, they apply to \textit{any} loop conditioned to be large in the rigid $\On$ loop model. To make it more concrete, on can choose any deterministic procedure to pick a loop in the rigid $\On$ loop model on quadrangulations of the sphere (e.g.\ the loop that is the closest to the root edge) and condition this loop to have perimeter $2k$. Then, the inner contour of this loop is the boundary of a loop-decorated quadrangulation with law $\Pnghk$ (see \cite{borot_recursive_2012,chen_perimeter_2017} for more details).

\medskip

\noindent\textbf{Overview and comments.} The paper is organized as follows. We first introduce a more general framework than \cite{le_gall_scaling_2011,borot_recursive_2012} for Boltzmann maps, by allowing slowly varying corrections. In Section \ref{sec:BoltzmannDistributionBoundary}, we extend known enumerative results to this family. The proofs of the main results rely on a decomposition of maps with a general boundary into a tree of maps with a simple boundary, inspired by \cite{curien_percolation_2014} and described in Section \ref{sec:StructureBoundary}. Then, we need enumerative results for bipartite maps with a simple boundary, which were unknown so far and are of independent interest. This is done in Section \ref{sec:BoltzmannDistributionSimpleBoundary}, by means of a second relation between maps with a general (resp.\ simple) boundary, and by using Tauberian theorems. This is a key feature of this work. 

This method is quite robust, and only needs estimates on the partition function of the model as an input. For this reason, we believe that our proofs can be adapted to more general statistical physics models on random maps for which Borot, Bouttier \& Guitter proved results similar to those of \cite{borot_recursive_2012}. For instance, general $\On$ loop models on triangulations with \textit{bending energy} \cite{borot_more_2012} or \textit{domain symmetry breaking} \cite{borot_loop_2012}. This last case covers in particular the \textit{Potts model} and \textit{Fortuin-Kasteleyn percolation} on general maps, that have been studied in \cite{berestycki_critical_2017,sheffield_quantum_2016,gwynne_scaling_2015-1,gwynne_scaling_2015,gwynne_scaling_2017,chen_basic_2017}. An interesting example is the critical Bernoulli percolation model on random triangulations, treated in \cite[Section 4.2, p.23]{borot_more_2012}. This corresponds to a $\On$ loop model on triangulations for $n=1$ and a suitable choice of the parameters. The asymptotics are similar to the quadrangular case, and we get the exponent $\beta=1/(1-\arccos(1/2)/\pi)=3/2$, which is consistent with the result of \cite{curien_percolation_2014}.

\begin{Not} Throughout this work, we use the notation $\N:=\{1,2,\ldots\}$ and $\Z_+:=\N\cup\{0\}$.\end{Not}

\section{Boltzmann distributions}\label{sec:BoltzmannLaws}

\subsection{Boltzmann distributions on bipartite maps}\label{sec:BoltzmannDistribution} 

\medskip

\noindent\textbf{Maps.} A \textit{planar map} is a proper embedding of a finite connected graph in the two-dimensional sphere $\mathbb{S}^2$, considered up to orientation-preserving homeomorphisms. The faces are the connected components of the complement of the embedding, and the degree $\deg(f)$ of a face $f$ is the number of its incident oriented edges. The sets of vertices, edges and faces of a (planar) map $\m$ are denoted by $\Vm$, $\Em$ and $\Fm$. For technical reasons, the maps we consider are always \textit{rooted}, which means that an oriented edge $e_*=(e_-,e_+)$, called the \textit{root edge}, is distinguished. The face $f_*$ incident on the right of the root edge is called the \textit{root face}. A map \textit{with a boundary} $\m$ is a map in which we consider the root face as an \textit{external face}, whose incident edges and vertices form the \textit{boundary} $\Bm$ of the map. The degree $\#\Bm$ of the external face is the \textit{perimeter} of the map, and non-root faces are called \textit{internal}.

In this paper, we only consider \textit{bipartite} maps, in which all face degrees are even. We denote by $\M$ the corresponding set, and by $\M_{k}$ be the set of (bipartite) maps with perimeter $2k$, for $k \geq 0$. The map $\dagger$ consisting of a single vertex is the only element of $\M_0$. We will also consider \textit{pointed} maps, which have a marked vertex $v_*$. A pointed bipartite map $\m$ such that $\dm (e_+,v_*) = \dm (e_-,v_*) +1$ is said to be \textit{positive}, and the corresponding set is denoted by $\Mp_+$ (by convention, $\dagger\in  \Mp_+$). Finally, $M$ stands for the identity mapping on $\M$.

\medskip

\noindent\textbf{Boltzmann distributions.} Given a \textit{weight sequence} $\q=(q_k : k\in \N)$ of nonnegative real numbers, the \textit{Boltzmann weight} of a bipartite map $\m$ is defined by
\begin{equation}\label{eqn:BoltzmannWeight}
	\wq(\m):=\prod_{f \in \Fm}q_{\deg(f)/2}.
\end{equation} By convention, we set $\wq(\dagger)=1$. This defines a $\sigma$-finite measure on $\Mp_+$ with total mass \begin{equation}\label{eqn:Zq}
	\Zq:=\wq\left(\Mp_+\right) \in [1,\infty].
\end{equation} A weight sequence $\q$ is \textit{admissible} if $\Zq<\infty$ (or equivalently if $\wq(\M)<\infty$, see \cite[Proposition 4.1]{bernardi_boltzmann_2017}). Then, the Boltzmann measure $\Pq$ is defined by
\[\Pq(\m):=\frac{\wq(\m)}{\Zq}, \quad \m\in\Mp_+.\] Following~\cite{marckert_invariance_2007}, we introduce the function 
\begin{equation}\label{eqn:fq}
	\fq(x):= \sum_{k=1}^{\infty}\binom{2k-1}{k-1}q_k x^{k-1}, \quad x\geq 0,
\end{equation} By \cite[Proposition 1]{marckert_invariance_2007}, a weight sequence $\q$ is admissible iff the equation 
	\begin{equation}\label{eqn:AdmissibilityQ}
		\fq(x)=1-\frac{1}{x}, \quad x>0
	\end{equation} has a solution. In that case, the smallest such solution is $\Zq$ and $\Zq^2\fq'(\Zq)\leq 1$. 
	
A classification of weight sequences was introduced in~\cite{marckert_invariance_2007,le_gall_scaling_2011}, which is closely related to the Bouttier-Di Francesco-Guitter bijection \cite{bouttier_planar_2004}. This bijection associates to every map $\m\in\Mp_+$ a plane tree $\BDG(\m)$ (together with labels on vertices at even height). The study is simplified by using additionally a bijection $\JS$ due to Janson and Stef\'ansson \cite[Section 3]{janson_scaling_2015}. This will be of independent interest, so we give a detailed presentation in Section \ref{sec:TreesAndJS}. We are interested in the application that associates to $\m\in \Mp_+$ the tree $\Phi(\m):=\JS(\BDG(\m))$. By \cite[Proposition 7]{marckert_invariance_2007} and \cite[Appendix A]{janson_scaling_2015} (see also Proposition \ref{prop:LawJSTransform}), we get the following.

\begin{Lem}\label{lem:JSLaw}
Let $\q$ be an admissible weight sequence. Under $\Pq$, the tree $\Phi(M)$ is a Galton-Watson tree with offspring distribution $\mu$ defined by
\[\mu(0)=1-\fq(\Zq) \quad \text{and} \quad \mu(k)=\Zq^{k-1}\binom{2k-1}{k-1}q_{k}, \quad k\in \N.\]
\end{Lem} Recall that the offspring distribution $\mu$ is critical (resp.\ subcritical) iff it has mean $\mmu=1$ (resp.\ $\mmu<1$). Lemma \ref{lem:JSLaw} transfers to the generating function $G_{\mu}$ of $\mu$, which reads 
\begin{equation}\label{eqn:GFmuAsfq}
	G_{\mu}(s):=\sum_{k= 0}^\infty s^k \mu(k)= 1-\fq(\Zq)+s\fq(s\Zq), \quad s\in[0,1].
\end{equation}

The aforementioned classification of weight sequences can be rephrased as follows. 

\begin{Def}\label{def:CriticityGenericity}
	An admissible sequence $\q$ is critical if $\mu$ is critical, and subcritical otherwise. A critical sequence $\q$ is generic critical if $\mu$ has finite variance. Finally, a critical sequence $\q$ is \textit{non-generic critical} with parameter $\alpha\in(1,2)$ if there exists a slowly varying function $\ell$ such that $\mu([k,\infty))=\ell(k)\cdot k^{-\alpha}$.
\end{Def}

Recall that a positive function $\ell$ is slowly varying (at infinity) if it satisfies $\ell(\lambda x)/\ell(x)\rightarrow 1$ as $x\rightarrow \infty$, for every $\lambda>0$. We emphasize that Definition \ref{def:CriticityGenericity} is more general than that of~\cite{le_gall_scaling_2011}, which implies that the slowly varying function $\ell$ is asymptotically constant (and is also the framework in~\cite{borot_recursive_2012,budd_geometry_2017,bertoin_martingales_2016,curien_peeling_2016}).

\begin{Rk}\label{rk:InterpretationCriticalityGenericity} The classification can be translated in terms of $\Pq$ by properties of $\BDG$. Namely, $\Eq(\#\mathrm{V}(M))=\infty$ iff $\q$ is critical and $\mub(k):=\mu(k+1)/\fq(\Zq)$ is interpreted as the law of (half) the degree of a typical face of the map under $\Pq$, see \cite{bouttier_planar_2004} for more on this.\end{Rk}

We conclude by translating Definition \ref{def:CriticityGenericity} in terms of the Laplace transform $L_{\mu}$ of $\mu$. First, if $\q$ is subcritical, $\mu$ has finite mean $\mmu<1$ and
\begin{equation}\label{eqn:GFMuSubcritical}
	L_{\mu}(t):=G_{\mu}(e^{-t})=1-\mmu t + o(t) \quad \text{as } t\rightarrow 0^+.
\end{equation} When $\q$ is generic critical, $\mu$ has mean $\mmu=1$ and finite variance $\smu$ which yields 
\begin{equation}\label{eqn:GFMuGenericCritical}
	L_{\mu}(t)=1- t + \frac{\smu+1}{2}t^2 +o(t^2) \quad \text{as } t\rightarrow 0^+.
\end{equation} For $\q$ non-generic critical with $\alpha \in (1,2)$, Karamata's Abelian theorem \cite[Theorem 8.1.6]{bingham_regular_1989} gives
\begin{equation}\label{eqn:GFMuNonGenericCritical}
	L_{\mu}(t)=1-t+\vert \Gamma(1-\alpha) \vert t^\alpha \ell\left(1/t\right)  + o(t^\alpha \ell\left(1/t\right)) \quad \text{as } t\rightarrow 0^+.
\end{equation} 

\subsection{Boltzmann distributions on maps with a boundary}\label{sec:BoltzmannDistributionBoundary}

We now deal with maps that have a boundary. The root face $f_*$ is then considered as external to the map, and receives no weight. This amounts to change the Boltzmann weights for
\begin{equation}\label{eqn:BoltzmannWeightBoundary}
	\wq(\m):=\prod_{f \in \Fm\backslash \{f_*\}}q_{\deg(f)/2}.
\end{equation} Let us introduce the partition functions for bipartite maps with a fixed perimeter
\begin{equation}\label{eqn:PartitionFunctionFk}
	F_k:=\sum_{\m\in\M_{k}}\wq(\m), \quad k\in \Z_+,
\end{equation} where we hide the dependence in $\q$ in the notation. These quantities are finite if $\q$ is admissible. The associated Boltzmann measure on maps with fixed perimeter is defined by
\begin{equation}\label{eqn:BoltzmannBoundary}
	\Pqk(\m):=\frac{\mathbf{1}_{\{\m \in \M_k \}}\wq(\m)}{F_k}, \quad \m\in\M, \ k\in \Z_+.
\end{equation} The goal of this section is to derive asymptotics of $F_k$. We also define the generating function \begin{equation}\label{eqn:GeneratingFunctionF}
	F(x):=\sum_{k=0}^\infty F_k x^k, \quad x\geq 0,
\end{equation} whose radius of convergence is denoted by $\rq$. We borrow ideas of \cite[Section 3.1]{borot_recursive_2012} and \cite[Section 5.1]{curien_peeling_2016}, but we need to extend these results due to our more general definition of non-genericity. We let the (admissible) weight sequence $\q$ vary by defining $\q(u):=(u^{k-1}\q_k : k\in \N)$. Using the generating function for pointed maps \cite[Proposition 2, Section A.1]{budd_peeling_2016} and Euler's formula, we obtain (see \cite[Equation (5.2)]{curien_peeling_2016})
\begin{equation}\label{eqn:IntegralFk}
	F_k=\binom{2k}{k} \int_0^1 (u\Zqu)^k \mathrm{d}u, \quad k\in \Z_+.
\end{equation} In the setting of \cite{le_gall_scaling_2011}, the asymptotics of $F_k$ would follow from Laplace's method, see \cite{borot_recursive_2012,curien_peeling_2016}. Here, we use a technique based on Karamata's Abelian theorem. Let $Y_\q$ be the inverse function of $u\mapsto u\Zqu$ on $[0,1]$. Since $\Zqu$ is the smallest solution of (\ref{eqn:AdmissibilityQ}) with $\q=\q(u)$ and $f_{\q(u)}(x)=\fq(ux)$, we have by (\ref{eqn:GFmuAsfq})
\[Y_{\q}(x)=x-x\fq(x)=1+x-\Zq G_{\mu}(x/\Zq), \quad x\in[0,\Zq].\] This proves that $Y_{\q}$ is of class $C^\infty$ on $(0,\Zq)$. Coming back to the integral in (\ref{eqn:IntegralFk}), 
\[\int_0^1 (u\Zqu)^k \mathrm{d}u=\int_0^{\Zq} x^kY_{\q}'(x) \mathrm{d}x=\Zq^{k+1}\int_0^\infty e^{-t(k+1)}Y_{\q}'(\Zq e^{-t}) \mathrm{d}t.\] We now introduce the increasing function 
\[U(t):=\int_0^t \Zq e^{-u}Y_{\q}'(\Zq e^{-u}) \mathrm{d}u=1-Y_{\q}(\Zq e^{-t})=-\Zq e^{-t} + \Zq L_{\mu}(t) , \quad t\geq 0.\] On the one hand, the integral is expressed in terms of the Laplace transform of $U$: 
\[\int_0^1 (u\Zqu)^k \mathrm{d}u=\Zq^k \int_0^\infty e^{-kt}U(\mathrm{d}t),\] and on the other hand from \eqref{eqn:GFMuSubcritical}, \eqref{eqn:GFMuGenericCritical} and \eqref{eqn:GFMuNonGenericCritical}, as $t\rightarrow 0^+$,
\begin{equation*}
U(t)=\left\lbrace
\begin{array}{ccc}
\Zq(1-\mmu)t+o(t) & (\q \ \textup{subcritical})\\
\\
\Zq \smu t^2/2 +o(t^2) & (\q \ \textup{generic critical})\\
\\
\Zq \vert \Gamma(1-\alpha) \vert t^\alpha \ell(1/t) +o(t^\alpha \ell(1/t)) & (\q \ \textup{non-generic critical } \alpha)\\
\end{array}\right..
\end{equation*} We can thus apply Karamata's Abelian theorem \cite[Theorem 1.7.1']{bingham_regular_1989}, giving
\begin{equation}\label{eqn:AsymptoticsFk}
F_k \underset{k\rightarrow \infty}{\sim} \left\lbrace
\begin{array}{ccc}
\displaystyle\frac{\Zq(1-\mmu)(4\Zq)^k}{\sqrt{\pi}k^{3/2}}   & (\q \ \textup{subcritical})\\
\\
\displaystyle\frac{\Zq\smu(4\Zq)^k}{\sqrt{\pi}k^{5/2}}   & (\q \ \textup{generic critical})\\
\\
\displaystyle\frac{\Zq\alpha\sqrt{\pi}(4\Zq)^k \ell(k)}{\sin(\pi(\alpha-1))k^{\alpha+1/2}} & (\q \ \textup{non-generic critical }\alpha)\\
\end{array}\right.,
\end{equation} where we used Stirling's formula and the identity $\Gamma(1-\alpha)\Gamma(1+\alpha)=\alpha\pi/\sin(\pi\alpha)$ for $\alpha \in (1,2)$. The quantity $a:=\alpha+1/2$ is of particular importance, so we use the notation of~\cite{curien_peeling_2016}.

\begin{Not}
An admissible weight sequence $\q$ is said of type $a=3/2$ if it is 	subcritical, of type $a=5/2$ if it is generic critical and of type $a\in(3/2,5/2)$ if it is non-generic critical with parameter $\alpha=a-1/2$.
\end{Not}

This allows us to write (\ref{eqn:AsymptoticsFk}) in a unified way. Let 
\begin{equation}\label{eqn:Constants}
	c_{3/2}:=\frac{\Zq(1-\mmu)}{\sqrt{\pi}}, \quad c_{5/2}:=\frac{\Zq\smu}{\sqrt{\pi}} \quad \textup{and} \quad c_{a}:=\frac{\Zq(a-1/2)\sqrt{\pi}}{\sin(\pi(a-3/2))},\quad a\in(3/2,5/2).
\end{equation} We also set the convention that $\ell=1$ if $a\in\{3/2,5/2\}$. Then,
\begin{equation}\label{eqn:AsymptoticsFkSimplified}
	F_k \underset{k\rightarrow \infty}{\sim}\frac{c_a (4\Zq)^k\ell(k)}{k^{a}}, \quad a\in[3/2,5/2].
\end{equation}

From now on, the case $a=2$ is excluded and will be treated apart in Section \ref{sec:TheA2Case}. Let us derive from \eqref{eqn:AsymptoticsFkSimplified} a singular expansion for $F$, whose radius of convergence is $\rq=1/(4\Zq)$. For $k\geq 0$, let
\[\zeta(k):=\frac{F_k \rq^k}{F(\rq)}\underset{k\rightarrow \infty}{\sim}\frac{c_a\ell(k)}{k^a F(\rq)}.\] The function $k\mapsto k^{a} \zeta(k)$ is slowly varying, so by Karamata's theorem~\cite[Proposition 1.5.10]{bingham_regular_1989}
\[\sum_{j\geq k}{\zeta(j)}\underset{k\rightarrow \infty}{\sim}\frac{k \zeta(k)}{a-1} \underset{k\rightarrow \infty}{\sim} \frac{c_a\ell(k)}{(a-1)F(\rq)k^{a-1}}.\] We then apply Karamata's Abelian theorem~\cite[Theorem 8.1.6]{bingham_regular_1989} to get the asymptotics of the Laplace transform $L_{\zeta}$ of $\zeta$. For $a\in[3/2,2)$, we find
\[L_{\zeta}(t)=1- \frac{\Gamma(2-a)c_a}{(a-1)F(\rq)} t^{a-1} \ell\left(1/t\right)  + o(t^{a-1}\ell\left(1/t\right)) \quad  \text{as } t\rightarrow 0^+,\] while for $a\in(2,5/2]$,
\[L_{\zeta}(t)=1-m_{\zeta}t+\frac{\vert \Gamma(2-a) \vert c_a}{(a-1)F(\rq)} t^{a-1} \ell\left(1/t\right)  + o(t^{a-1}\ell\left(1/t\right)) \quad \text{as } t\rightarrow 0^+.\] The function $\ell_1(y):=\ell(-1/\log(1-1/y))$ is slowly varying at infinity by stability properties of slowly varying functions~\cite[Proposition 1.3.6]{bingham_regular_1989}. We obtain that for $a\in[3/2,2)$, 
\[G_{\zeta}(s)=1- \frac{\Gamma(2-a)c_a}{(a-1)F(\rq)} (1-s)^{a-1} \ell_1\left(\frac{1}{1-s}\right)(1+o(1)) \quad  \text{as } s\rightarrow 1^-,\] and for $a\in(2,5/2]$,
\[G_{\zeta}(s)=1-m_{\zeta}(1-s)+\frac{\vert \Gamma(2-a) \vert c_a}{(a-1)F(\rq)} (1-s)^{a-1} \ell_1\left(\frac{1}{1-s}\right)(1+o(1))\quad  \text{as } s\rightarrow 1^-.\] The singular expansion of $F$ follows from $F(x\rq)=F(\rq)G_{\zeta}(x)$. Note that  we have $m_{\zeta}=\rq F'(\rq)/F(\rq)$, and let $\kappa_a:=c_a\vert \Gamma(2-a) \vert/(a-1)$. Recall also that $ \ell_1=1$ for $a\in\{3/2,5/2\}$.

\begin{Prop}\label{prop:SingularExpansionF} Let $\q$ be a weight sequence of type $a$. For $a\in[3/2,2)$,
\[F(x)=F(\rq)- \kappa_a\left(1-\frac{x}{\rq}\right)^{a-1} \ell_1\left(\frac{1}{1-\frac{x}{\rq}}\right)(1+o(1))\quad  \text{as } x\rightarrow \rq^-,\] and for $a\in(2,5/2]$,
\[F(x)=F(\rq)-\rq F'(\rq)\left(1-\frac{x}{\rq}\right)+\kappa_a \left(1-\frac{x}{\rq}\right)^{a-1} \ell_1\left(\frac{1}{1-\frac{x}{\rq}}\right)(1+o(1))\quad  \text{as } x\rightarrow \rq^-.\]
\end{Prop}

\subsection{Boltzmann distributions on maps with a simple boundary}\label{sec:BoltzmannDistributionSimpleBoundary}

The aim of this section is to obtain enumerative results for maps with a simple boundary. A (bipartite) map with a \textit{simple boundary} is a map whose boundary is a cycle with no self-intersection. Their set is denoted by $\Ms$. Consistently, for $k \geq 0$, $\Ms_{k}$ is the set of maps with a simple boundary of perimeter $2k$. A generic element of $\Ms$ is denoted by $\ms$, and $\dagger\in\Ms_0$ by convention. The associated partition function is
\begin{equation}\label{eqn:PartitionFunctionFks}
	\Fs_k:=\sum_{\ms\in\Ms_{k}}\wq(\ms), \quad k\in \Z_+.
\end{equation} For admissible $\q$, the Boltzmann measure for maps with a simple boundary is defined by
\begin{equation}\label{eqn:BoltzmannSimpleBoundary}
	\Pqks(\m):=\frac{\mathbf{1}_{\left\lbrace \m \in \Ms_k \right\rbrace }\wq(\m)}{\Fs_k}, \quad \m\in\M, \ k\in \Z_+,
\end{equation} and the associated generating function by
\begin{equation}\label{eqn:GeneratingFunctionFs}
	\Fs(x):=\sum_{k=0}^\infty \Fs_k x^k \quad x\geq 0.
\end{equation} The radius of convergence of $\Fs$ is denoted by $\rsq$. We will prove the following analogue of Proposition \ref{prop:SingularExpansionF} for maps with a simple boundary, which is the technical core of this paper. The constants $(\cs_a : a\in\{3/2\}\cup (2,5/2])$ and the slowly varying functions $\ls_1$ (also depending on $a$) will be defined at the end of the section, see \eqref{eqn:ConstantsHat} and \eqref{eqn:FunctionsEllHat}.

\begin{Prop}\label{prop:SingularExpansionFs} Let $\q$ be a weight sequence of type $a$. For $a=3/2$, as $y\rightarrow \rq F^2(\rq)^-$,
\[\Fs(y)=F(\rq)\left( 1-\frac{1}{2}\left(1-\frac{y}{\rq F^2(\rq)}\right) +\cs_{3/2}\left(1-\frac{y}{\rq F^2(\rq)}\right)^2(1+o(1)) \right).\] If $a\in (3/2,5/2]\backslash\{2\}$, $\Fs$ has radius of convergence $\rsq=\rq F^2(\rq)$. Moreover, for $a\in(3/2,2)$,
\[\Fs(y)=F(\rq)\left( 1-\frac{1}{2}\left(1-\frac{y}{\rsq}\right) +\left(1-\frac{y}{\rsq}\right)^{\frac{1}{a-1}}\ls_1\left(\frac{1}{1-\frac{y}{\rsq}} \right)(1+o(1)) \right) \quad \text{as } y\rightarrow \rsq^-,\]and for $a\in (2,5/2]$,
\[\Fs(y)=F(\rq)\left( 1-\frac{\cs_a}{2}\left(1-\frac{y}{\rsq}\right) +\left(1-\frac{y}{\rsq}\right)^{a-1}\ls_1\left(\frac{1}{1-\frac{y}{\rsq}} \right)(1+o(1)) \right) \quad \text{as } y\rightarrow \rsq^-.\]
	
\end{Prop}

Our approach relies on a simple relation between the generating functions $F$ and $\Fs$, which was first observed in~\cite{brezin_planar_1978} (see also \cite{bouttier_distance_2009} for quadrangulations). This relation is based on the decomposition of a map with a boundary $\m$ into a map with a simple boundary $\ms$ containing the root edge, and a collection of maps with a general boundary attached to vertices of $\partial\ms$ (see \cite[Figure 11 and Equation (5.1)]{bouttier_distance_2009} for details). We then obtain the following identity.

\begin{Lem}[\cite{brezin_planar_1978,bouttier_distance_2009}]\label{lem:RelationFFs}
For every weight sequence $\q$ and every $x\geq 0$, we have
\[F(x)=\Fs\left(xF^2(x)\right).\] In particular, the radius of convergence of $\Fs$ satisfies $\rsq\geq\rq F^2(\rq)$.
\end{Lem}

We now use this relation to prove Proposition \ref{prop:SingularExpansionFs}. For $x\geq 0$, let $P(x):=xF^2(x)$, so that $P$ is continuous increasing on $[0,\rq]$ with inverse $P^{-1}$. By Proposition \ref{prop:SingularExpansionF}, for $a\in[3/2,2)$, \begin{equation}\label{eqn:AsymptoticPa1}
	P(x)=P(\rq)- \kappa'_a\left(1-\frac{x}{\rq}\right)^{a-1} \ell_1\left(\frac{1}{1-\frac{x}{\rq}}\right)(1+o(1))\quad  \text{as } x\rightarrow \rq^-,
\end{equation} and for $a\in(2,5/2]$,
\begin{equation}\label{eqn:AsymptoticPa2}
	P(x)=P(\rq)-C_{\q}\left(1-\frac{x}{\rq}\right)+\kappa'_a \left(1-\frac{x}{\rq}\right)^{a-1} \ell_1\left(\frac{1}{1-\frac{x}{\rq}}\right)(1+o(1))\quad  \text{as } x\rightarrow \rq^-,
\end{equation}where $C_{\q}:=\rq F(\rq)(F(\rq)+2\rq F'(\rq))$ and $\kappa'_a:=2\rq F(\rq)\kappa_a$. We now invert this expansion to get that of $P^{-1}$, and treat $a\in[3/2,2)$ and $a\in(2,5/2]$ separately. Recall that a positive function $f$ is regularly varying (at infinity) with index $\gamma\in\R$ if it satisfies $f(\lambda x)/f(x)\rightarrow \lambda^\gamma$ as $x\rightarrow\infty$, for every $\lambda>0$. The next lemma is a variant of~\cite[Theorem 1.5.12]{bingham_regular_1989}.

\begin{Lem}\label{lem:RegularlyVaryingInversion}
	Let $f$ be a continuous decreasing regularly varying function with index $-\gamma<0$. Then, $f$ is invertible and the function $y\mapsto f^{-1}(1/y)$ is regularly varying with index $1/\gamma$.\end{Lem}

Let $a\in[3/2,2)$. From (\ref{eqn:AsymptoticPa1}), we know that $R(x):=P(\rq)-P(\rq(1-1/x)) \sim \kappa'_a x^{1-a} \ell_1(x)$ as $x\rightarrow \infty$, thus $R$ is regularly varying with index $1-a<0$. Moreover, $R$ is continuous decreasing on $[1,\infty)$ with inverse $R^{-1}$ defined by
\[R^{-1}(y)=\left. 1 \middle/ \left(1-\frac{1}{\rq}P^{-1}(P(\rq)-y)\right)\right., \quad y\in(0,P(\rq)].\] By Lemma \ref{lem:RegularlyVaryingInversion}, $y\mapsto R^{-1}(1/y)$ is regularly varying with index $1/(a-1)$, so that~\cite[Theorem 1.4.1]{bingham_regular_1989} ensures the existence of a positive slowly varying function $\lb_1$ such that $R^{-1}(1/y)=y^{1/(a-1)}\lb_1(y)$, for $y\in[1/P(\rq),\infty)$. As a consequence, 
\begin{equation}\label{eqn:AsymptoticInversePDense}
	P^{-1}(y)=\rq- \rq \left(P(\rq)-y \right)^{\frac{1}{a-1}}\left.\middle/ \left(\lb_1\left(\frac{1}{P(\rq)-y} \right)\right)\right., \quad y\in [0,P(\rq)).
\end{equation} When $a=3/2$, $\ell_1=1$ so that computation can be made more explicit. Indeed, we find $R(x)\sim \kappa'_{3/2}/\sqrt{x}$ as $x\rightarrow \infty$. Then, the function $Q(x):=R((\kappa'_{3/2}/x)^2)$ satisfies $Q^{-1}(y)\sim y$ as $y\rightarrow 0^+$ and \[R^{-1}(y)=\left(\frac{\kappa'_{3/2}}{Q^{-1}(y)}\right)^2\sim \left(\frac{\kappa'_{3/2}}{y}\right)^2 \quad  \text{as } y\rightarrow 0^+.\] As a conclusion,
\begin{equation}\label{eqn:AsymptoticInversePSubcritical}
	P^{-1}(y)=\rq-\frac{\rq}{(\kappa'_{3/2})^2}\left(P(\rq)-y\right)^2(1+o(1)) \quad \text{as } y\rightarrow P(\rq)^-.
\end{equation} 

We are now interested in the case where $a\in(2,5/2]$. From (\ref{eqn:AsymptoticPa2}), we have
\[R(x):=\left.\left[P(\rq)-P(\rq(1-x)) \right]\middle/ C_{\q}\right. =x-\frac{\kappa'_a}{C_{\q}} x^{a-1} \ell_1\left(1/x\right)(1+o(1))\quad  \text{as } x\rightarrow 0^+.\] The function $R$ is continuous increasing on $[0,1]$, with inverse $R^{-1}$ defined by
\[R^{-1}(y)=1-\frac{1}{\rq}P^{-1}(P(\rq)-C_{\q}y), \quad y\in[0,P(\rq)/C_{\q}].\] It also satisfies $R^{-1}(y)\sim y$ as $y\rightarrow 0^+$. In particular, $y\mapsto R^{-1}(1/y)$ is regularly varying with index $-1$ and by \cite[Proposition 1.5.7]{bingham_regular_1989}, $\lb_1(y):=\ell_1(1/R^{-1}(1/y))$ is slowly varying. We get 
\[R^{-1}(y)-y\sim \frac{\kappa'_a}{C_{\q}} \left(R^{-1}(y)\right)^{a-1} \ell_1\left(1/R^{-1}(y)\right)\sim \frac{\kappa'_a}{C_{\q}} y^{a-1} \lb_1\left(1/y\right)\quad  \text{as } y\rightarrow 0^+,\] and as a conclusion 
\begin{equation}\label{eqn:AsymptoticInversePDilute}
	P^{-1}(y)=\rq-\frac{\rq }{C_{\q}}\left(P(\rq)-y\right)-\frac{\kappa'_a}{C_{\q}^a}\left(P(\rq)-y\right)^{a-1}\lb_1\left(\frac{C_{\q}}{P(\rq)-y}  \right)(1+o(1)) \ \ \text{as } y\rightarrow P(\rq)^-.
\end{equation} We can now introduce the constants involved in the statement of Proposition \ref{prop:SingularExpansionFs},
\begin{equation}\label{eqn:ConstantsHat}
	\cs_{3/2}:=\frac{P(\rq)^2}{2(\kappa'_{3/2})^2}-\frac{1}{8} \quad \text{and} \quad \cs_a=1-\frac{P(\rq)}{C_{\q}}\in(0,1) \quad \text{for } a\in(2,5/2].
\end{equation} and the functions $\ls_1$ (that are slowly varying by~\cite[Proposition 1.3.6]{bingham_regular_1989}) defined by
\begin{equation}\label{eqn:FunctionsEllHat}
 \ls_1(y):=\frac{P(\rq)^{\frac{1}{a-1}}}{2\lb_1\left( \frac{y}{P(\rq)} \right)},\quad a\in(3/2,2) \quad \text{and} \quad \ls_1(y):=\frac{\kappa'_a P(\rq)^{a-1}}{2C_{\q}^a}\lb_1\left( \frac{C_{\q}y}{P(\rq)} \right), \quad a\in(2,5/2].
\end{equation} 

\begin{proof}[Proof of Proposition \ref{prop:SingularExpansionFs}.] By Lemma \ref{lem:RelationFFs}, we have that $\Fs(\rq F^2(\rq))=F(\rq)$, as well as \[\Fs(y)=\sqrt{\frac{y}{P^{-1}(y)}}, \quad 0<y\leq P(\rq).\] We obtain asymptotic expansions for $\Fs$ around $P(\rq)$ using \eqref{eqn:AsymptoticInversePDense}, \eqref{eqn:AsymptoticInversePSubcritical}, and \eqref{eqn:AsymptoticInversePDilute}.
	These expansions are singular for $a\neq 3/2$, and thus $\Fs$ is not of class $C^{\infty}$ at $P(\rq)$. Together with Lemma \ref{lem:RelationFFs}, this proves that the radius of convergence of $\Fs$ is $\rsq=P(\rq)$ in these cases.\end{proof}

\begin{Rk}\label{rk:QuadrangularCase} From there, one expects the theory of singularity analysis \cite[Chapter 6]{flajolet_analytic_2009} to give an equivalent of the partition function $\Fs_k$. However, it is not clear that the so-called delta-analyticity assumption is satisfied by $\Fs$. We will use instead Karamata's Tauberian theorem, which provides a weaker result (see Proposition \ref{prop:PropertyNu}). 

Note also that in the subcritical case, we do not know if $\rsq=\rq F^2(\rq)$ in general because the expansion of $\Fs$ is not singular. In the special case of quadrangulations, computations can be carried out explicitly using~\cite{bouttier_distance_2009}. We find that $\rsq>\rq F^2(\rq)$ if $\q$ is  subcritical. Moreover,
\[ \Fs_k\underset{k \rightarrow \infty}{\sim} \frac{2\sqrt{3}\rsq^{-k}}{27\sqrt{\pi}k^{5/2}} \quad (\q \text{ critical}) \quad \text{and} \quad  \Fs_k\underset{k \rightarrow \infty}{\sim} \frac{c_\q\rsq^{-k}}{k^{3/2}} \quad (\q \text{ subcritical}).\]
\end{Rk}

\section{Structure of the boundary of Boltzmann maps}\label{sec:StructureBoundary}

\subsection{Random trees and the Janson-Stef\'ansson bijection}\label{sec:TreesAndJS}

\medskip

\noindent\textbf{Trees.} A (finite) plane tree $\tr$ \cite{le_gall_random_2005,neveu_arbres_1986} is a finite subset of the sequences of positive integers
\[\mathcal{U}:=\bigcup_{n\in \Z_+}{\mathbb{N}^n}\] satisfying the following properties. First, $\emptyset \in \tr$ and is called the \textit{root vertex}. Then, for every $u=(u_1,\ldots,u_k) \in \tr$, $\uh:=(u_1,\ldots,u_{k-1}) \in \tr$ (and is called the \textit{parent} of $u$ in $\tr$). Finally, for every $u=(u_1,\ldots,u_k) \in \tr$, there exists $k_u=k_u(\tr)\in\Z_+$ (the number of children of $u$ in $\tr$) such that $uj:=(u_1,\ldots,u_{k},j)\in \tr$ iff $1\leq j \leq k_u$.  The height $\vert u \vert$ of a vertex $u=(u_1,\ldots,u_k)\in\tr$ is $\vert u \vert=k$, and we denote by $[\emptyset,u]$ (resp.\ $[\emptyset,u)$) the ancestral line of $u$ in $\tr$, $u$ included (resp.\ excluded). The vertices at even height are called white, and those at odd height are called black. We let $\tr_{\circ}$ and $\tr_{\bullet}$ be the corresponding subsets of vertices of $\tr$. The total number of vertices of a tree $\tr$ is denoted by $\vert\tr\vert$. The set of finite (plane) trees is denoted by $\T_f$, and $T$ stands for the identity mapping on $\T_f$.

Given a probability measure $\rho$ on $\Z_+$ with mean $\mrho\leq 1$, the law $\GWr$ of a \textit{Galton-Watson tree} with offspring distribution $\rho$ is characterized by 
\begin{equation}\label{eqn:GW}
	\GWr(\tr)=\prod_{u\in \tr}\rho(k_u), \quad \forall \ \tr\in\T_f.
\end{equation} A pair $(\rhow,\rhob)$ of probability measures on $\Z_+$ is called critical (resp.\ subcritical) if $\mrhow\mrhob=1$ (resp.\ $\mrhow\mrhob<1$). Then, the law $\GWrr$ of an \textit{(alternated) two-type Galton-Watson trees} with offspring distribution $(\rhow,\rhob)$ is characterized by
\begin{equation}\label{eqn:GWTwoType}
	\GWrr(\tr)=\prod_{u\in \tw}\rhow(k_u)\prod_{u\in \tb}\rhob(k_u), \quad \forall \ \tr\in\T_f.
\end{equation}

\medskip

\noindent\textbf{The Janson-Stef\'ansson bijection.} We now describe the Janson-Stef\'ansson bijection $\JS$ introduced in~\cite[Section 3]{janson_scaling_2015}. First, $\JS(\{\emptyset\})=\{\emptyset\}$. For $\tr\neq \{\emptyset\}$, $\JS(\tr)$ has the same vertices as $\tr$ but different edges defined as follows. For every $u\in\tw$, set the convention that $u0=\uh$ (if $u\neq\emptyset$) and $u(k_u+1)=u$. Then, for every $j\in\{0,1,\ldots,k_u\}$, add the edge $(uj,u(j+1))$ to $\JS(\tr)$. The root vertex of $\JS(\tr)$ is $1$ and its first children is chosen according to the lexicographical order of $\tr$. For further notice, we give a brief description of the inverse application $\JS^{-1}$. For $\tr\neq \{\emptyset\}$, $\JS^{-1}(\tr)$ has the same vertices as $\tr$, and edges defined as follows. For every leaf $u\in\tr$, let $(u_1,u_2,\ldots)$ be the sequence of vertices after $u$ in the contour order of $\tr$, and $\ell(u)$ the largest index such that $u_1,\ldots,u_{\ell(u)}$ all are ancestors of $u$ in $\tr$. Then, add an edge between $u$ and $u_k$ in $\JS^{-1}(\tr)$ for every $k\in\{1,\ldots,\ell(u)\}$. The last leaf $u'$ of $\tr$ in contour order is the root vertex, and $u'_{\ell(u)}$ its first child. 

The application $\JS$ is a bijection from $\T_f$ onto itself, such that every $u\in\tw$ is mapped to a leaf of $\JS(\tr)$, and every $u\in \tb$ with $k$ children is mapped to a vertex of $\JS(\tr)$ with $k+1$ children. See Figure \ref{fig:JS} for an illustration. This bijection simplifies the study of two-type Galton-Watson trees because of the following result of~\cite{janson_scaling_2015} (see also~\cite[Proposition 3.6]{curien_percolation_2014}).

\begin{Prop}{\textup{\cite[Appendix A]{janson_scaling_2015}}}\label{prop:LawJSTransform} Let $\rho_{\circ}$ and $\rho_{\bullet}$ be probability measures on $\Z_+$ such that $m_{\rho_\circ}m_{\rho_\bullet}\leq 1$ and $\rho_{\circ}$ has geometric distribution with parameter $1-p\in(0,1)$: $\rho_{\circ}(k)=(1-p)p^k$ for $k\geq 0$. Then, the image of $\mathsf{GW}_{\rho_{\circ},\rho_{\bullet}}$ under $\JS$ is $\mathsf{GW}_{\rho}$, where
\[\rho(0)=1-p \quad \text{and} \quad \rho(k)=p\cdot\rho_{\bullet}(k-1), \quad k\in \N.\] In particular, $m_\rho-p=(1-p)m_{\rho_\circ}m_{\rho_\bullet}$, so that $(\rho_{\circ},\rho_{\bullet})$ is critical iff $\rho$ itself is critical.
\end{Prop}

\begin{figure}[ht]
	\centering
	\includegraphics[scale=.8]{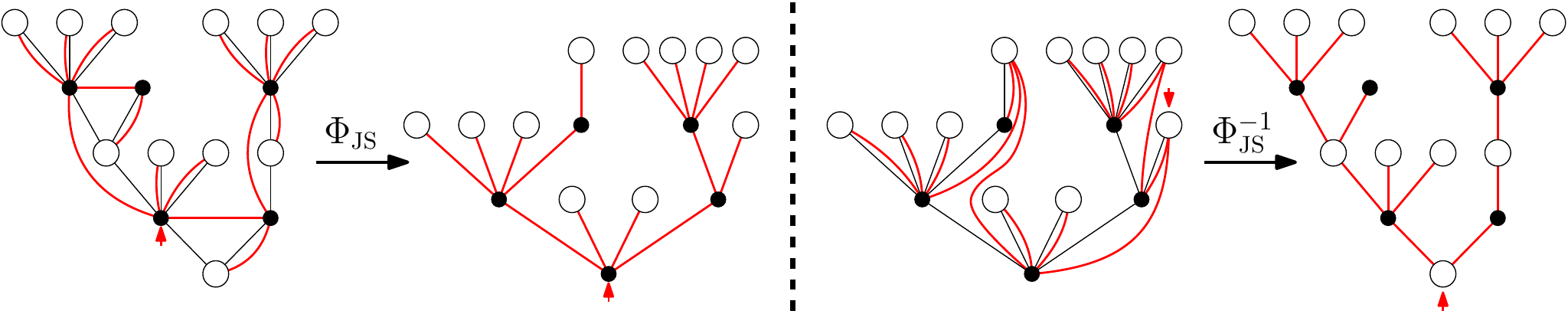}
	\caption{The Janson-Stef\'ansson bijection and its inverse application.}
	\label{fig:JS}
	\end{figure}

\subsection{Random looptrees and scooped-out maps}\label{sec:RandomLooptrees}

We now introduce random looptrees and their \textit{tree of components} to represent the boundary of a map as a tree, following the presentation of~\cite[Section 2.3]{curien_percolation_2014} (see also \cite{curien_random_2014}).  

\medskip

\noindent\textbf{Random looptrees.} A looptree is a map whose edges are incident to two distinct faces, one being the root face (such a map is called edge-outerplanar). Informally, a (finite) looptree is a collection of polygons glued along a tree structure. Their set is denoted by $\LT_f$.

We associate to every tree $\tr\in \T_f$ a looptree $\Loop(\tr)$ as
follows. For every $u\in\tb$, connect all the incident (white) vertices of $u$ in cyclic order. Then,
$\Loop(\tr)$ is the map obtained by discarding the black vertices and edges of $\tr$. The root edge of $\Loop(\tr)$ connects the origin
of $\tr$ to the last child of its first offspring. The inverse application associates to every looptree $\lt\in \LT_f$ a tree $\Tree(\lt)$, called the tree of components, as follows. We add an extra vertex in every internal face of $\lt$, which we connect by an edge to all the vertices of this face. The tree $\Tree(\lt)$ is obtained by discarding the edges of $\lt$. The root edge of $\Tree(\lt)$ connects the origin of $\lt$ to the vertex inside the internal face incident to the root. See Figure~\ref{fig:TreeAndLoop} for an example. 

\begin{figure}[ht]
  \centering
    \includegraphics[scale=.55]{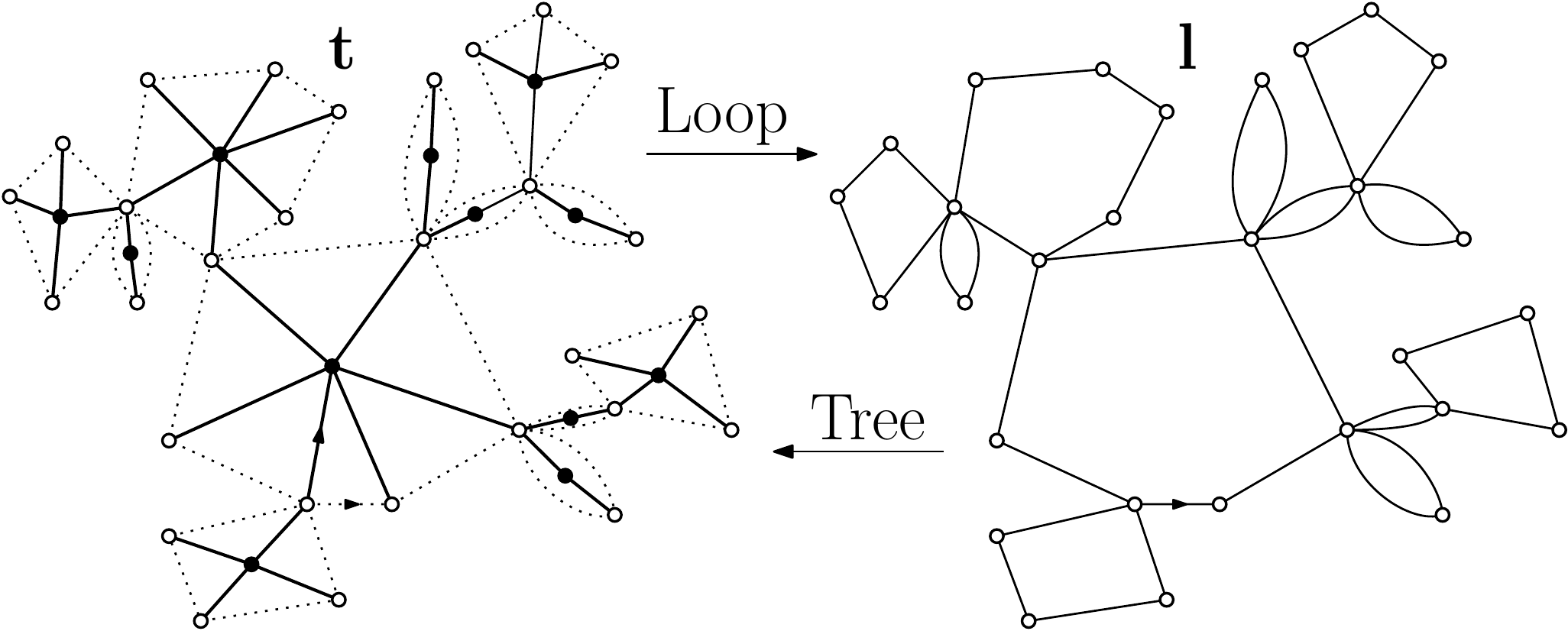}
  \caption{A looptree $\lt$ and the associated tree of components $\tr$.}
  \label{fig:TreeAndLoop}
\end{figure}

\begin{Rk}\label{rk:GluingLooptrees} Every internal face of $\lt\in\LT_f$ is rooted at the oriented edge whose origin is the closest to that of $\lt$, and such that the root face lies on its right. The gluing of a map with a simple boundary of perimeter $k$ into a face of degree $k$ is then determined by the convention that the root edges match.\end{Rk}

This definition of looptree slightly differs from that of \cite{curien_random_2014,curien_percolation_2014}, that we now recall. Given a tree $\tr\in\T_f$, the looptree $\Loopb(\tr)$ (or $\Loopb'(\tr)$ in \cite{curien_random_2014}) is built from $\tr$ as follows. For every $u,v\in \tr$, there is an edge between $u$ and $v$ iff one of these conditions is fulfilled: $u$ and $v$ are consecutive siblings in $\tr$, or $v$ is either the first or the last child of $u$ in $\tr$. We will also need $\Loopbb(\tr)$, which is obtained from $\Loopb(\tr)$ by contracting the edges linking a vertex of $\tr$ and its last child in $\tr$. These objects are rooted at the oriented edge between the origin of $\tr$ and its last child in $\tr$ (resp.\ penultimate for $\Loopbb$). See \cite[Figures 9 and 10]{curien_percolation_2014} for an example. We use the bold print $\Loopb$ to distinguish this construction from $\Loop$. Note that contrary to $\Loop$, $\Loopb$ does not allow several loops to be glued at the same vertex.

\medskip

\noindent\textbf{The scooped-out map.} The \textit{scooped-out} map of a map $\m$ was defined in \cite{curien_percolation_2014} as the looptree $\Scoop(\m)$ obtained from $\Bm$ by duplicating the edges whose both sides belong to the root face. We call tree of components of $\m$ the tree $\Treeb(\m):=\Tree(\Scoop(\m))$.

A map $\m$ is recovered from $\Scoop(\m)$ by gluing into its internal faces the proper maps with a simple boundary. These maps are the connected components obtained when splitting $\m$ at the pinch-points of $\partial\m$, called \textit{irreducible components} in \cite{bouttier_distance_2009} and \cite[Section 2.2]{curien_uniform_2015}. They have the same rooting convention as in Remark \ref{rk:GluingLooptrees}. This construction provides a bijection 
\[\TC: \m \mapsto \left(\Treeb(\m), \left(\ms_u : u \in \Treeb(\m)_\bullet\right)\right) \] that associates to a map $\m\in\M$ the tree $\tr=\Treeb(\m)$, whose vertices at odd height have even degree, and a collection $(\ms_u : u\in\tb)$ of maps with a simple boundary of respective perimeter $\deg(u)$. See Figure \ref{fig:ScoopedOutMap} for an example. The following relations will be useful:
\begin{equation}\label{eqn:PerimeterAndSizeOfTree}
	\vert \tr \vert = \#\Bm + 1 \quad \text{and} \quad \sum_{u\in \tb }{\deg(u)}=\#\Bm  \quad (\tr=\Treeb(\m)).
\end{equation}

\begin{figure}[ht]
	\centering
	\includegraphics[scale=.65]{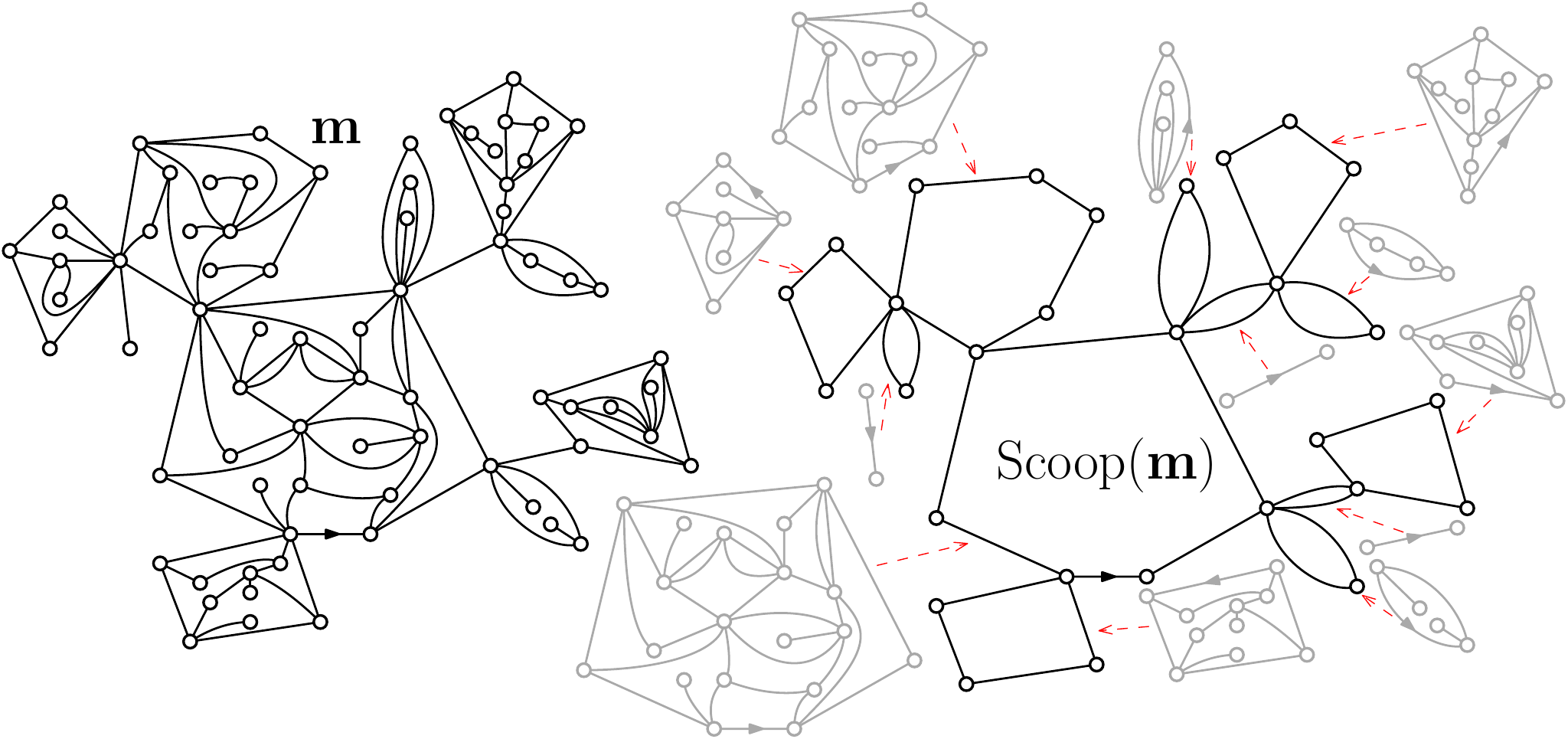}
	\caption{A planar map $\m$ and the associated scooped-out map $\Scoop(\m)$.}
	\label{fig:ScoopedOutMap}
	\end{figure}

\subsection{Distribution of the tree of components}\label{sec:LawTreeOfComponents}

We now introduce the probability measure $\Pqr$ with ``free perimeter" defined by
\begin{equation}\label{eqn:DefPqr}
	\Pqr(\m):=\frac{\rq^{\#\Bm/2}\wq(\m)}{F(\rq)}, \quad \m\in\M.
\end{equation} It is related to $\Pqk$ by conditioning with respect to the perimeter of the map: for every $k\geq 0$ and $\m\in\M$, we have $\Pqr\left(\m \mid \M_k\right)=\Pqk(\m)$. The main result of this section identifies the distribution of the tree of components (see also \cite[Proposition 6]{baur_uniform_2016} for quadrangulations).

\begin{Prop}\label{prop:LawTreeComponents}Let $\q$ be a weight sequence of type $a\in [3/2,5/2]$. Under $\Pqr$, $\Treeb(M)$ is a two-type Galton-Watson tree with offspring distribution $(\nuw,\nub)$ defined by
\[\nuw(k)=\frac{1}{F(\rq)}\left(1-\frac{1}{F(\rq)}\right)^k \quad \text{and} \quad \nub(2k+1)=\frac{1}{F(\rq)-1}\left(\rq F^2(\rq)\right)^{k+1}\Fs_{k+1}, \quad k\in \Z_+.\] (With $\nub(2\Z_+)=0$.) Moreover, conditionally on $\Treeb(M)$, the maps with a simple boundary $(\Mh_u : u\in \Treeb(M)_\bullet)$ associated to $M$ by $\TC$ are independent with respective law $\widehat{\mathbb{P}}^{(\deg(u)/2)}_{\q}$.
\end{Prop}

\begin{proof}Let us check that $\nuw$ and $\nub$ are probability measures. This is clear for $\nuw$, and since $1<F(\rq)=\Fs(\rq F^2(\rq))$ we get
\[ \sum_{k\in \Z_+}\nub(k)=\frac{1}{F(\rq)-1}\left(\Fs(\rq F^2(\rq))-1\right)=1.\] Recall that $\TC$ associates to $\m\in\M$ its tree of components $\tr=\Treeb(\m)$ and maps $(\ms_u : u\in \tb)$ with a simple boundary of perimeter $\deg(u)$. Using \eqref{eqn:DefPqr} and (\ref{eqn:PerimeterAndSizeOfTree}), we have
\begin{align*}
	\Pqr(\m)=\frac{\rq^{\#\Bm/2}\wq(\m)}{F(\rq)}=\frac{1}{F(\rq)}\prod_{u\in\tb}\rq^{\deg(u)/2}\wq(\ms_u).
\end{align*} Then, for every $c>0$
\[1=\prod_{u\in\tw}{c^{k_u}\left(\frac{1}{c}\right)^{\vert \tb\vert}} \quad \text{and}  \quad \frac{1}{c}=\prod_{u\in\tb}{c^{k_u}\left(\frac{1}{c}\right)^{\vert \tw \vert}}. \] Applying the first identity with $c=1-1/F(\rq)$ and the second one with $c=F(\rq)$ yields
\begin{align*}
	\Pqr(\m)&=\prod_{u\in\tw}{\frac{1}{F(\rq)}\left(1-\frac{1}{F(\rq)}\right)^{k_u}}\prod_{u\in\tb}{\frac{1}{F(\rq)-1}\left(\rq F^2(\rq)\right)^{(k_u+1)/2}\wq(\ms_u)}\\
	&=\prod_{u\in\tw}{\nuw(k_u)}\prod_{u\in\tb}{\nub(k_u)\wq(\ms_u)\frac{1}{\Fs_{(k_u+1)/2}}}.\end{align*} By convention, both sides equal zero if there exists $u\in\tb$ such that $\Fs_{(k_u+1)/2}=0$. Finally,
\[\Pqr\left(\Treeb(M)=\tr, \Mh_u=\ms_u : u\in \tb\right)=\Pqr(M=\m)=\GWnn(\tr) \prod_{u\in\tb}\widehat{\mathbb{P}}^{(\deg(u)/2)}_{\q}(\ms_u),\] which is the expected result.\end{proof}

By Proposition \ref{prop:LawJSTransform}, we obtain the following.

\begin{Cor}\label{cor:LawTreeComponentsJS}Let $\q$ be a weight sequence of type $a\in [3/2,5/2]$. Under $\Pqr$, $\JS(\Treeb(M))$ is a Galton-Watson tree with offspring distribution $\nu$ defined by
\[\nu(2k)=\frac{1}{F(\rq)}\left(\rq F^2(\rq)\right)^{k}\Fs_{k}, \quad k\in \Z_+ \quad \text{(and } \nu(k)=0 \text{ for } k \text{ odd).}\]\end{Cor}

As a consequence, the generating function of $\nu$ reads
\begin{equation}\label{eqn:GFnu}
	G_{\nu}(s)=\frac{1}{F(\rq)}\sum_{k=0}^\infty s^{2k}\left(\rq F^2(\rq) \right)^{k}\Fs_{k}=\frac{1}{F(\rq)}\Fs\left(\rq F^2(\rq) s^2\right), \quad s\in[0,1].
\end{equation} From Lemma \ref{lem:RelationFFs}, we easily deduce the following formula for the mean of $\nu$ 
\begin{equation}\label{eqn:MeanNu}
	\mnu=G'_{\nu}(1)=\frac{1}{F(\rq)}2\rq F^2(\rq) \Fs'\left(\rq F^2(\rq)\right) = \frac{1}{1+\frac{F(\rq)}{2\rq F'(\rq)}}.
\end{equation}  
Similarly, the generating function of $\nub$ satisfies $G_{\nub}(0)=0$ and 
\begin{equation}\label{eqn:GFnub}
	G_{\nub}(s)=\frac{1}{F(\rq)-1}\cdot\frac{1}{s}\left(\Fs\left(\rq F^2(\rq) s^2\right)-1\right), \quad s\in(0,1].
\end{equation} The next result is a consequence of (\ref{eqn:AsymptoticsFkSimplified}), \eqref{eqn:MeanNu} and Proposition \ref{prop:LawJSTransform}.

\begin{Lem}\label{lem:CriticityJSTree}
	The offspring distribution $\nu$ and the pair of offspring distributions $(\nuw,\nub)$ are critical if $a\in [3/2,2)$ and subcritical if $a\in(2,5/2]$.
\end{Lem}

We now describe $\nu$ and $\nub$ using Proposition \ref{prop:SingularExpansionFs}, \eqref{eqn:GFnu} and \eqref{eqn:GFnub}. For $a=3/2$, as $t\rightarrow 0^+$
\begin{align}\label{eqn:LaplaceNu1}
  L_{\nu}(t)&=1-t+\left(1+4\cs_{3/2}\right)t^2+o(t^2),\\\label{eqn:LaplaceNub1}
  L_{\nub}(t)&=1-\left(\frac{1}{F(\rq)-1}\right)t+\left(\frac{1}{2} +4\cs_{3/2}\frac{F(\rq)}{F(\rq)-1}\right)t^2+o(t^2).
\end{align} For $a\in(3/2,2)$, as $t\rightarrow 0^+$
\begin{align}\label{eqn:LaplaceNu2}
  L_{\nu}(t)&=1-t+2^{\frac{1}{a-1}}t^{\frac{1}{a-1}}\ls(1/t) +o\left(t^{\frac{1}{a-1}}\ls(1/t)\right),\\\label{eqn:LaplaceNub2}
  L_{\nub}(t)&=1-\left(\frac{1}{F(\rq)-1}\right)t+\frac{F(\rq)}{F(\rq)-1}2^{\frac{1}{a-1}}t^{\frac{1}{a-1}}\ls(1/t) +o\left(t^{\frac{1}{a-1}}\ls(1/t)\right).
\end{align}Finally, for $a\in (2,5/2]$, as $t\rightarrow 0^+$,
\begin{align}\label{eqn:LaplaceNu3}
  L_{\nu}(t)&=1-\cs_a t+2^{a-1}t^{a-1}\ls(1/t) +o\left(t^{a-1}\ls(1/t)\right),\\\label{eqn:LaplaceNub3}
  L_{\nub}(t)&=1-\left(1-\cs_a \frac{F(\rq)}{F(\rq)-1}\right)t+\frac{F(\rq)}{F(\rq)-1}2^{a-1}t^{a-1}\ls(1/t) +o\left(t^{a-1}\ls(1/t)\right).
\end{align} The function $\ls(x):=\ls_1(1/(1-\exp(-2/x)))$ is slowly varying from \cite[Proposition 1.5.7]{bingham_regular_1989}. For $a=3/2$, \eqref{eqn:LaplaceNu1} and (\ref{eqn:ConstantsHat}) entail that $\nu$ and $\nub$ have finite variance equal to
\begin{equation}\label{eqn:ConclusionNu1}
	\snu=\left(\frac{2P(\rq)}{\kappa'_{3/2}}\right)^2=\left(\frac{F(\rq)}{\Zq(1-\mmu)}\right)^2 \quad \text{and} \quad \snub=\frac{F(\rq)}{F(\rq)-1}\left(\left(\frac{F(\rq)}{\Zq(1-\mmu)}\right)^2-1\right).
\end{equation} For $a\in(3/2,2)$, Karamata's Tauberian theorem~\cite[Theorem 8.1.6]{bingham_regular_1989}, (\ref{eqn:LaplaceNu2}) and \eqref{eqn:LaplaceNub2} give\begin{equation}\label{eqn:ConclusionNu2}
	\nu([k,\infty))\underset{k\rightarrow \infty}{\sim} \frac{2^{\frac{1}{a-1}}}{\left\vert \Gamma\left(\frac{a-2}{a-1}\right) \right\vert} \cdot \frac{\ls(k)}{k^{\frac{1}{a-1}}} \quad \text{and} \quad \nub([k,\infty))\underset{k\rightarrow \infty}{\sim} \frac{F(\rq)}{F(\rq)-1}\cdot\frac{2^{\frac{1}{a-1}}}{\left\vert \Gamma\left(\frac{a-2}{a-1}\right) \right\vert} \cdot \frac{\ls(k)}{k^{\frac{1}{a-1}}}.
\end{equation} Finally, when $a\in(2,5/2]$, the same version of Karamata's Tauberian theorem gives
\begin{equation}\label{eqn:ConclusionNu3}
	\nu([k,\infty))\underset{k\rightarrow \infty}{\sim} \frac{2^{a-1}}{\vert \Gamma\left(2-a\right)\vert} \cdot \frac{\ls(k)}{k^{a-1}} \quad \text{and} \quad \nub([k,\infty))\underset{k\rightarrow \infty}{\sim} \frac{F(\rq)}{F(\rq)-1}\cdot\frac{2^{a-1}}{\vert \Gamma\left(2-a\right)\vert} \cdot \frac{\ls(k)}{k^{a-1}}.
\end{equation} 

\begin{Prop}\label{prop:PropertyNu}
	For $a=3/2$, $\nu$ and $\nub$ have finite variance $($and exponential moments iff $\rsq>\rq F^2(\rq)$$)$. For $a\in(3/2,2)$, $\nu$ and $\nub$ are in the domain of attraction of a stable law with parameter $1/(a-1)\in (1,2)$ and for $a\in(2,5/2]$, $\nu$ and $\nub$ are in the domain of attraction of a stable law with parameter $a-1\in(1,3/2]$.
\end{Prop}

For every $n\geq 1$, let $\GWr^{(n)}$ (resp.\ $\GWrr^{(n)}$) be the law of a Galton-Watson tree with offspring distribution $\rho$ (resp.\ $(\rhow,\rhob)$) conditioned to have $n$ vertices, provided this makes sense. We have the following conditioned version of Proposition \ref{prop:LawTreeComponents} and Corollary \ref{cor:LawTreeComponentsJS}.

\begin{Cor}\label{cor:LawConditionedTreeComponents} Let $\q$ be a weight sequence of type $a\in[3/2,5/2]$. Under $\Pqk$, $\Treeb(M)$ has law $\GWnn^{(2k+1)}$, and $\JS(\Treeb(M))$ has law $\GWn^{(2k+1)}$. Moreover, conditionally on $\Treeb(M)$, the maps $(\Mh_u : u\in \Treeb(M)_\bullet)$ associated to $M$ by $\TC$ are independent with law $\widehat{\mathbb{P}}^{(\deg(u)/2)}_{\q}$.
\end{Cor}

\section{Scaling limits of the boundary of Boltzmann maps}\label{sec:ScalingLimitsSection}

This section deals with the scaling limits of the boundary of Boltzmann maps in the Gromov-Hausdorff sense. We refer to~\cite{burago_course_2001} for a complete definition of this topology. We start with a preliminary result directly adapted from \cite[Lemma 4.3]{curien_percolation_2014}.

\begin{Lem}{\textup{\cite{curien_percolation_2014}}}\label{lem:ScoopAndLoop} For every $\m\in\M$, we have $\Scoop(\m)=\Loopbb(\JS(\Treeb(\m)))$.
\end{Lem}

\medskip

\noindent\textbf{Scaling limits: the dense regime.} We first focus on the dense phase $a\in(3/2,2)$ and prove Theorem \ref{th:ScalingDense}. The proof parallels that of \cite[Theorem 1.2]{curien_percolation_2014}.

\begin{proof}[Proof of Theorem \ref{th:ScalingDense}.] For every $k\geq 0$, let $M_k$ be a random map with law $\Pqk$ and set $T_k:=\JS(\Treeb(M_k))$. By definition of $\Loopbb$, we have
\begin{equation}\label{eqn:GHDistanceLoop}
	\dgh\left(\Loopb(T_k),\Loopbb(T_k)\right)\leq 2 H(T_k),
\end{equation} where $H(T_k)$ is the overall height of $T_k$. Indeed, the longest path of vertices of $T_k$ that are identified in $\Loopbb(T_k)$ has length at most $H(T_k)$. By scaling limits results for conditioned Galton-Watson trees (\cite[Theorem 3.1]{duquesne_limit_2003}, \cite[Theorem 3]{kortchemski_simple_2013}) we have that
\begin{equation}\label{eqn:CvgProbaHeight}
	\frac{H(T_k)}{k^{a-1}}\underset{k\rightarrow \infty}{\longrightarrow} 0 \quad \text{in probability.} 
\end{equation} The results of \cite{duquesne_limit_2003,kortchemski_simple_2013} together with \eqref{eqn:CvgProbaHeight} ensure that the invariance principle of \cite[Theorem 4.1]{curien_random_2014} applies: there exists a slowly varying function $\Lambda$ such that in the Gromov-Hausdorff sense
\[\frac{\Lambda(k)}{(2k)^{a-1}}\cdot \Loopb(T_k) \underset{k\rightarrow \infty}{\overset{(d)}{\longrightarrow}} \Lts_{\frac{1}{a-1}}. \] Applying \eqref{eqn:GHDistanceLoop}, \eqref{eqn:CvgProbaHeight} and Lemma \ref{lem:ScoopAndLoop}, we deduce that in the Gromov-Hausdorff sense
\[\frac{\Lambda(k)}{(2k)^{a-1}} \cdot \Scoop(M_k) \underset{k\rightarrow \infty}{\overset{(d)}{\longrightarrow}} \Lts_{\frac{1}{a-1}}.\] This concludes the proof since $\Bm$ and $\Scoop(\m)$ always define the same metric space.\end{proof}

\begin{Rk}\label{rk:ConstantOn} In the setting of \cite{le_gall_scaling_2011,borot_recursive_2012} (in particular for applications to the $\On$ model), the definition of non-generic critical sequences imply that $\Lambda$ can be replaced by a constant.
\end{Rk}

\medskip

\noindent\textbf{Scaling limits: the subcritical regime.} In the subcritical case, we expect that there exists $K_\q>0$ such that in the Gromov-Hausdorff sense
\begin{equation*}
	\frac{K_\q}{\sqrt{2k}}\cdot \partial M_k \underset{k\rightarrow \infty}{\overset{(d)}{\longrightarrow}} \Trs_\e, 
\end{equation*} where $\Trs_\e$ is the Continuum Random Tree \cite{aldous_continuum_1991,aldous_continuum_1993}. When $\nu$ has exponential moments (i.e., if $\rsq>\rq F^2(\rq)$) this follows from \cite[Theorem 14]{curien_crt_2015}. As mentioned in Remark \ref{rk:QuadrangularCase}, we do not know if this is satisfied for all subcritical sequences. However, we believe that \cite[Theorem 14]{curien_crt_2015} holds under a finite variance assumption, by proving tightness of the sequence of laws and identifying the finite-dimensional marginals.

\medskip

\noindent\textbf{Scaling limits: the generic and dilute regimes.} In the generic and dilute regimes, we believe that there exists $K_\q>0$ such that in the Gromov-Hausdorff sense
\begin{equation*}
	\frac{K_\q}{2k}\cdot \partial M_k \underset{k\rightarrow \infty}{\overset{(d)}{\longrightarrow}} \mathbb{S}_1,
\end{equation*} where $\mathbb{S}_1$ stands for the unit circle. A proof could be adapted from \cite[Theorem 1.2]{curien_percolation_2014}, which is itself based on the results of \cite{jonsson_condensation_2010,kortchemski_limit_2015} about condensation in non-generic trees. However, these results apply only if we have an equivalent of the partition function $\Fs_k$, which our techniques do not provide (see Remark \ref{rk:QuadrangularCase}).

\section{Local limits of the boundary of Boltzmann maps}\label{sec:LocalLimitsSection}

\subsection{Local limits of Galton-Watson trees}\label{sec:LocalLimitGW}

\medskip

\noindent\textbf{The local topology.} The \textit{local topology} on the set $\M$ is induced by the local distance \begin{equation}\label{eqn:DefLocalCvg}
	\dloc(\m,\m'):=\left. 1 \middle/ \left(1+\sup\left\lbrace R\geq 0 : \B_R(\m) = \B_R(\m')\right\rbrace \right)\right., \quad \m,\m' \in \M.
\end{equation} Here, $\B_R(\m)$ is the ball of radius $R$ in $\m$ for the graph distance, made of all the vertices of $\m$ at distance less than $R$ from the origin, and all the edges whose endpoints are in this set. We let $\M’$ be the completed space of $\M$, so that elements of $\M_\infty := \M’ \backslash \M$ are infinite (bipartite) maps. All the elements of $\M_\infty$ we consider can be seen as proper embeddings of a graph in the plane (up to orientation preserving homeomorphisms, see \cite[Proposition 2]{curien_peeling_2016}). Then, the boundary $\partial\m$ of $\m\in\M_\infty$ is the embedding of edges and vertices of its root face. When the boundary is infinite, it is called simple if isomorphic to $\Z$. 

In order to take account of convergence towards plane trees with vertices of infinite degree, a weaker form of local convergence has been introduced in \cite{jonsson_condensation_2010} (see also \cite[Section 6]{janson_simply_2012}). The idea is to replace the ball $\B_R(\tr)$ in \eqref{eqn:DefLocalCvg} by the sub-tree $\Bl_R(\tr)$, called the \textit{left ball} of radius $R$ of $\tr$. Formally, the root vertex belongs to $\Bl_R(\tr)$, and a vertex $u=\uh k\in\tr$ belongs to $\Bl_R(\tr)$ iff $\uh\in\tr$, $k\leq R$ and $\vert u \vert \leq R$. 

For our purposes, a slightly stronger form of convergence is needed. For every $\tr \in \T_f$ and every $u\in\tr$, we denote by $(-u1,-u2,\ldots,-uk_u)=(uk_u,u(k_u-1),\ldots,u1)$ the children of $u$ in counterclockwise order. For every $\tr\in \T_f$ and every $R\geq 0$, the \textit{left-right ball} of radius $R$ in $\tr$ is the sub-tree $\Blr_R(\tr)$ defined as follows. First, $\emptyset\in\Blr_R(\tr)$. Then, a vertex $u\in\tr$ belongs to  $\Blr_R(\tr)$ iff $\uh\in \Blr_R(\tr)$, $\vert u \vert \leq 2R$ and $u\in\{\uh 1, \ldots, \uh R\}\cup \{-\uh 1, \ldots, -\uh R\}$ ($u$ is among the $R$ first or last children of its parent). We call \textit{local-$*$ topology} the topology on $\T_f$ induced by
\[\dloc^*(\tr,\tr'):=\left. 1 \middle/\left(1+\sup\left\lbrace R\geq 0 : \Blr_R(\tr) = \Blr_R(\tr') \right\rbrace \right)\right., \quad \tr,\tr' \in \T_f.\] The set $\T$ of general trees is the completion of $\T_f$ for $\dloc^*$, while the set $\Tloc$ of locally finite trees is the completion of $\T_f$ for $\dloc$. 

\medskip

\noindent\textbf{Local limits of conditioned Galton-Watson trees.} We next recall results concerning local limits of Galton-Watson trees conditioned to survive. 

\smallskip

\noindent\textit{The critical case.} The critical setting was first investigated by Kesten \cite{kesten_subdiffusive_1986} (see also \cite{abraham_local_2014}) and extended by Stephenson in \cite{stephenson_local_2016}. Let $(\rhow,\rhob)$ be a critical pair of offspring distributions, and recall that for every probability measure $\rho$ on $\Z_+$ with mean $\mrho\in(0,\infty)$, the size-biased distribution $\bar{\rho}$ is defined by 
\[\bar{\rho}(k):=\frac{k\rho(k)}{\mrho}, \quad k\in \Z_+.\] The infinite random tree $\Tinfwb=\Tinfwb(\rhow,\rhob)$ is defined as follows. It has a.s.\ a unique spine, in which white (resp.\ black) vertices have offspring distribution $\bar{\rho}_\circ$ (resp.\ $\bar{\rho}_{\bullet}$), and a unique child in the spine chosen uniformly among their offspring. Outside of the spine, white (resp.\ black) vertices have offspring distribution $\rhow$ (resp.\ $\rhob$), and all the numbers of offspring are independent. The tree $\Tinfwb$ is illustrated in Figure \ref{fig:InfiniteLooptrees}.

\begin{Prop}{\textup{\cite[Theorem 3.1]{stephenson_local_2016}}}\label{prop:ConvergenceKestenTree} Let $(\rhow,\rhob)$ be a critical pair of offspring distributions. For every $k\geq 1$, let $\Twb_k$ be a tree with law $\GWrr^{(k)}$. Then, we have in the local sense
	\[\Twb_k \underset{k\rightarrow \infty}{\overset{(d)}{\longrightarrow}} \Tinfwb(\rhow,\rhob).\]
\end{Prop} Here and after, we implicitly work along a subsequence on which $\GWrr(\{\vert \tr \vert =k\})>0$.

\smallskip

\noindent\textit{The subcritical case.} We start with subcritical monotype trees, first considered in \cite{jonsson_condensation_2010} and studied in full generality in \cite{janson_simply_2012,abraham_local_2014-1}. Let $\rho$ be a subcritical offspring distribution (such that $\rho(0)\in (0,1)$). The infinite random tree $\Tinf=\Tinf(\rho)$ is defined as follows. It has a.s.\ a unique finite spine of random size $L$, such that $P(L=k)=(1-\mrho)\mrho^{k-1}$ for $k\in\N$. The last vertex of the spine has infinite degree. The $L-1$ first vertices of the spine have offspring distribution $\bar{\rho}$, and a unique child in the spine chosen uniformly among the offspring. Outside of the spine, vertices have offspring distribution $\rho$, and all the numbers of offspring are independent. This defines a random element of $\T$.

\begin{Prop}\label{prop:ConvergenceJSTree} Let $\rho$ be a subcritical offspring distribution with no exponential moment $($and $\rho(0)\in (0,1)$$)$. For every $k\geq 1$, let $T_k$ be a tree with law $\GWr^{(k)}$. Then, in the local-$*$ sense,
	\[T_k \underset{k\rightarrow \infty}{\overset{(d)}{\longrightarrow}} \Tinf(\rho).\]
\end{Prop} 

\begin{proof}The proof follows from \cite[Theorem 7.1]{janson_simply_2012}. However, this result is equivalent to the convergence of left-balls of any radii (see \cite[Lemma 6.3]{janson_simply_2012}), which is weaker than our statement. Then, observe that for every $\tr\in\T_f$, $k\geq 0$ and $R\geq 0$ we have 
	\[\GWr^{(k)}\left( \Blr_R(T)=\tr \right)=\GWr^{(k)}\left( \Bl_{2R}(T)=\tr \right).\] Indeed, $\GWr^{(k)}$ is invariant under the operation consisting in exchanging the descendants of $(u(R+1),\ldots u(2R))$ and $(-u1,\ldots -uR)$ for every $u\in \tr$ such that $k_u(\tr)>2R$ (which exchanges $\Blr_R(\tr)$ and $\Bl_{2R}(\tr)$). This concludes the argument.\end{proof}
 
We now extend Proposition \ref{prop:ConvergenceJSTree} to two-type Galton-Watson trees. Let $(\rhow,\rhob)$ be a subcritical pair of offspring distributions. We build a two-type version $\Tinfwb=\Tinfwb(\rhow,\rhob)$ of $\Tinf$ as follows. It has a.s.\ a unique spine, with random number of vertices $2L'$ satisfying \[P(L'=k)=(1-\mrhow\mrhob)(\mrhow\mrhob)^{k-1}, \quad k\in\N.\] In the spine, the topmost (black) vertex has infinite degree, while other vertices have offspring distribution $\bar{\rho}_\circ$ (if white) and $\bar{\rho}_\bullet$ (if black), with a unique child in the spine chosen uniformly among the offspring. Outside of the spine, white (resp.\ black) vertices have offspring distribution $\rhow$ (resp.\ $\rhob$), and all the numbers of offspring are independent (see Figure~\ref{fig:InfiniteLooptrees}).

\begin{Prop}\label{prop:ConvergenceJSTreeTwoType} Let $(\rhow,\rhob)$ be a subcritical pair of offspring distributions such that $\rhow$ is geometric with parameter $1-p\in(0,1)$ $(\rhow(k)=(1-p)p^k$ for $k\geq 0)$, and $\rhob$ has no exponential moment. For every $k\geq 1$, let $\Twb_k$ be a tree with law $\GWrr^{(k)}$. Then, in the local-$*$ sense,
	\[\Twb_k \underset{k\rightarrow \infty}{\overset{(d)}{\longrightarrow}} \Tinfwb(\rhow,\rhob).\]
\end{Prop}

\begin{proof}For every $k\geq 1$, let $T_k:=\JS(\Twb_k)$. By Proposition \ref{prop:LawJSTransform}, $T_k$ has law $\GWr^{(k)}$, where 
\[\rho(0)=1-p \quad \text{and} \quad \rho(k)=p\cdot\rho_{\bullet}(k-1), \quad k\in \N.\] In particular, $\rho$ satisfies the hypothesis of Proposition \ref{prop:ConvergenceJSTree}. For every $N\geq 1$, let $u_N=u_N(T_k)$ be the first vertex of $\Blr_N(T_k)$ in contour order having $2N$ offspring (or the root vertex otherwise). For every $R\geq 0$, we also let $T_k \langle u_N,R \rangle$ be the collection of subtrees of $T_k$ containing all the children of $u_N$ different from $\{\pm u_N1,\ldots \pm u_NR\}$, as well as their descendants. Finally, set $T_k[N,R]:=\Blr_N(T_k)\backslash T_k \langle u_N,R \rangle$, and extend these definitions to $\Tinf=\Tinf(\rho)$. We denote by $u_\infty$ the vertex of infinite degree of $\Tinf$, and let $\Tinf[R]$ be the subtree of $\Tinf$ in which children of $u_\infty$ other than $\{u_\infty 1,\ldots u_\infty R\}$ and their descendants are discarded. This definition extends to $\Tinfwb$.

Fix $R\geq 0$. By Proposition \ref{prop:ConvergenceJSTree} and the definition of $\Tinf$, we have in the local sense
\begin{equation}\label{eqn:PrunedJSTree}
	T_k[N,R+1]\underset{k\rightarrow \infty}{\overset{(d)}{\longrightarrow}} \Tinf[N,R+1], \quad \text{and} \quad \Tinf[N,R+1]\underset{N\rightarrow \infty}{\overset{(d)}{\longrightarrow}} \Tinf[2(R+1)].
\end{equation} In particular, the event (measurable with respect to $\Blr_N(T_k)$)
\[\mathcal{E}(R,N,k):=\left\lbrace \sup\{\vert u \vert \vee k_u : u\in T_k[N,R+1]\} <N\right \rbrace\] has probability tending to one when $k$ and then $N$ go to infinity. On the event $\mathcal{E}(R,N,k)$, one has $T_k\backslash T_k[N,R+1]\subseteq T_k \langle u_N,R+1 \rangle $, which in turn enforces \begin{equation}\label{eqn:StabilityJSEventE}
	\Blr_R(\Twb_k)=\Blr_R(\JS^{-1}(T_k))\subseteq \JS^{-1}(T_k[N,R+1]).
\end{equation} Indeed, on this event, the images of vertices of $T_k\backslash T_k[N,R+1]$ in $\JS^{-1}(T_k)$ are descendants of the children of $u'_N:=\JS^{-1}(u_N)$ that are not in $\{\pm u'_N1,\ldots \pm u'_NR\}$. (See Figure \ref{fig:ProofLocal}.)

Let $d\geq 0$, and keep the notation $u_\infty$ for the pointed vertex with $d$ children in $\Tinf[d]$ and $\Tinfwb[d]$. We let $\GWr^{[d]}$ be the law of $(\Tinf[d],u_\infty)$, and $\GWrr^{[d]}$ be that of $(\Tinfwb[d],u_\infty)$. Then,
\begin{equation}\label{eqn:PushforwardJSGW}
	\JS\left(\GWrr^{[d]}\right)=\GWr^{[d+1]}.
\end{equation} We temporarily admit \eqref{eqn:PushforwardJSGW} and conclude the proof. Let $A$ be a Borel set for the local-$*$ topology. We have by \eqref{eqn:StabilityJSEventE} that for every $k\geq 1$ and $N\geq 1$
\[\left| P\left(\Blr_R\left( \Twb_k \right)\in A \right)-P\left(\Blr_R\left( \JS^{-1}(T_k[N,R+1]) \right)\in A \right) \right|\leq 2P(\mathcal{E}(R,N,k)^c).\] Next, for every $N\geq 1$, \eqref{eqn:PrunedJSTree} entails \[\left| P\left(\Blr_R\left( \JS^{-1}(T_k[N,R+1]) \right)\in A \right)-P\left(\Blr_R\left( \JS^{-1}(\Tinf[N,R+1]) \right)\in A \right) \right|\underset{k\rightarrow \infty}{\longrightarrow} 0.\] Then, by \eqref{eqn:PrunedJSTree} again and the fact that $\Tinf[2(R+1)]$ is a.s.\ finite,
\[\left| P\left(\Blr_R\left( \JS^{-1}(\Tinf[N,R+1]) \right)\in A \right)-P\left(\Blr_R\left( \JS^{-1}(\Tinf[2(R+1)]) \right)\in A \right) \right| \underset{N\rightarrow \infty}{\longrightarrow} 0.\]  Finally, for every $R\geq 0$, $\Blr_R\left( \Tinfwb \right)=\Blr_R\left( \Tinfwb[2R+1]\right)$ by definition so that by \eqref{eqn:PushforwardJSGW}, \[P\left(\Blr_R\left( \JS^{-1}(\Tinf[2(R+1)]) \right)\in A \right)=P\left(\Blr_R\left( \Tinfwb[2R+1] \right)\in A \right)=P\left(\Blr_R\left( \Tinfwb \right)\in A \right).\] As a conclusion, by letting $k$ and then $N$ go to infinity, we have 
\begin{align*}\lim_{k\rightarrow \infty}\left| P\left(\Blr_R\left( \Twb_k \right)\in A \right)-P\left(\Blr_R\left( \Tinfwb \right)\in A \right) \right|=0.
\end{align*} Let us now prove assertion \eqref{eqn:PushforwardJSGW}. Let $(\tr,u^*)$ be a pointed plane tree such that $k_{u^*}(\tr)=d+1$. By definition, $(\tr',v^*):=\JS^{-1}(\tr,u^*)$ is a pointed plane tree satisfying $k_{v^*}(\tr')=d$, and $v^*\in \tb'$. Then, we have by definition of $\rhow$ and the identity $\sum_{u\in\tw'}{k_u(\tr')}=\vert \tb'\vert$,
\begin{align*}\GWrr^{[d]}\left(\JS^{-1}((\tr,u^*))\right)=\GWrr^{[d]}\left((\tr',v^*)\right)&=\frac{1-\mrhow\mrhob}{\mrhow}\prod_{u\in \tw'}{(1-p)p^{k_u(\tr')}} \prod_{u\in \tb' \atop u \neq v^* }{\rhob(k_u(\tr'))}\\
&=\frac{p(1-\mrhow\mrhob)}{\mrhow}\prod_{u\in \tw'}{(1-p)} \prod_{u\in \tb' \atop u \neq v^* }{p\cdot \rhob(k_u(\tr'))}.
\end{align*} Vertices of $\tw'$ are mapped to leaves of $\tr$ by $\JS$, while vertices of $\tb'$ with $k$ children are mapped to vertices of $\tr$ with $k+1$ children. By Proposition \ref{prop:LawJSTransform}, we get
\begin{align*}\GWrr^{[d]}\left(\JS^{-1}((\tr,u^*))\right)
	&=(1-\mrho)\prod_{u\in \tr \atop k_u(\tr)=0 }{(1-p)} \prod_{u\in \tr\backslash\{u^*\} \atop k_u(\tr)>0 }{p\cdot \rhob(k_u(\tr)-1)}\\
	&=(1-\mrho)\prod_{u\in \tr \backslash\{u^*\}}{\rho(k_u(\tr))},
\end{align*} which is $\GWr^{[d+1]}((\tr,u^*))$, as expected.\end{proof}

\begin{figure}[ht]
	\centering
	\includegraphics[scale=1.1]{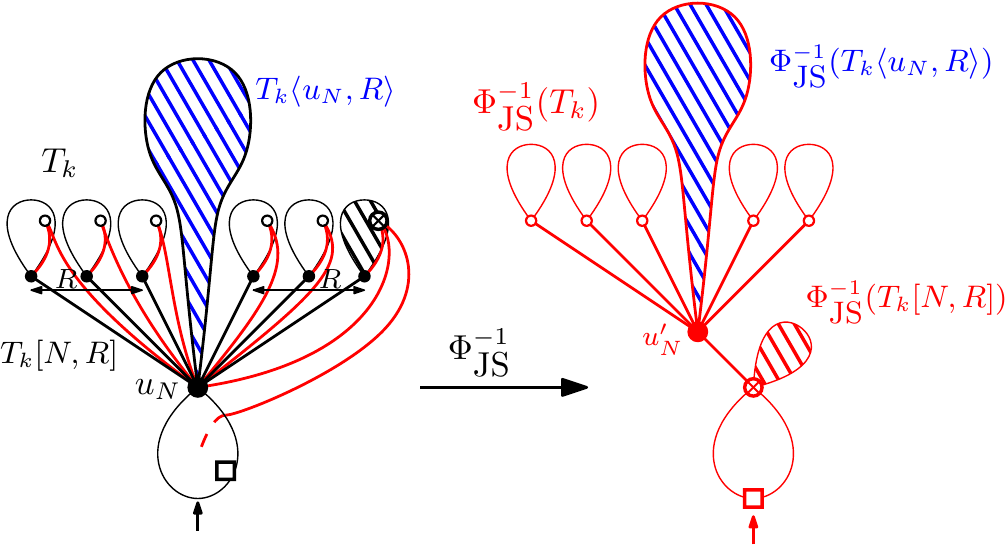}
	\caption{The image of $T_k$ by $\JS^{-1}$, on the event $\mathcal{E}(R,N,k)$. The boxed vertex is the last leaf of $T_k$ in contour order, while the crossed vertex is the last leaf among the descendants of $u_N$.}
	\label{fig:ProofLocal}
	\end{figure}

We conclude with a property of $\Tinfwb$ in the subcritical case. Let $u_\infty$ be the unique vertex with infinite degree of $\Tinfwb$, and $\uh_\infty$ its parent. There exists $j\in\{1,\ldots,k_{\uh_\infty}\}$ such that $u_\infty=\uh_\infty j$. We define the vertex $u_\infty^\leftarrow$ as $\uh_\infty (j-1)$ if $j>1$, and $\uh_\infty$ itself if $j=1$. The vertex $u_\infty$ and its incident edges disconnect $\Tinfwb$ in infinitely many connected components that we denote by $(\Tr_i:i\in \Z)$. For every $i\neq 0$, $\Tr_i$ is the connected component containing $u_\infty i$, rooted at the oriented edge going from $u_\infty i$ to its first child in $\Tinfwb$. Finally, $\Tr_0$ is the connected component containing the root vertex of $\Tinfwb$, and has the same root edge as $\Tinfwb$. 

\begin{Lem}\label{lem:LawSizeBiasedGW}
	The plane trees $(\Tr_i:i\in \Z)$ are independent. For every $i\neq 0$, $\Tr_i$ has law $\GWrr$, while $\Tr_0$ has the size-biased law $\overline{\GWrr}$ defined by
	\[\overline{\GWrr}(\tr)=\frac{\vert \tr \vert \GWrr(\tr)}{\GWrr(\vert T \vert)}, \quad \tr \in \T_f.\] Moreover, conditionally on $\Tr_0$, $u_\infty^\leftarrow$ has uniform distribution on $\Tr_0$.
\end{Lem}

\begin{proof}We focus on $\Tr_0$. Let $(\tr,u^*)$ be a pointed plane tree, and let $u^\circ$ be either the parent of $u^*$ in $\tr$ if $u^*\in \tb$, or $u^*$ itself otherwise. Then, $(\Tr_0,u_\infty^\leftarrow)=(\tr,u^*)$ enforces $\uh_\infty=u^\circ$. Since $k_{\uh_\infty}(\Tr_0)=k_{\uh_\infty}(\Tinfwb)-1$ and by definition of $\rhow$, we obtain
\begin{align*}P\left((\Tr_0,u_\infty^\leftarrow))=(\tr,u^*)\right)&=\prod_{u\in \tw \atop u \in [\emptyset,u^\circ) }{\frac{\bar{\rho}_\circ(k_u(\tr))}{k_u(\tr)}} \prod_{u\in \tb \atop u \in [\emptyset,u^\circ)}{\frac{\bar{\rho}_\bullet(k_u(\tr))}{k_u(\tr)}}\prod_{ u\in \tw \atop u \notin [\emptyset,u^\circ]  }{\rhow(k_u(\tr))} \prod_{u\in \tb \atop u \notin [\emptyset,u^\circ] }{\rhob(k_u(\tr))}\\
&\qquad\qquad\qquad\times \bar{\rho}_\circ(k_{u^\circ}(\tr)+1)\frac{1}{k_{u^\circ}(\tr)+1} (1-\mrhow\mrhob)(\mrhow\mrhob)^{\frac{\vert u^\circ \vert}{2}}\\
&=\frac{p(1-\mrhow\mrhob)}{\mrhow}\prod_{u\in \tw}{\rhow(k_u(\tr))} \prod_{u\in \tb}{\rhob(k_u(\tr))}=(1-\mrho)\GWrr(\tr).
\end{align*} We conclude by Proposition \ref{prop:LawJSTransform}, which gives $\GWrr(\vert T\vert)=\GWr(\vert T\vert)=1/(1-\mrho)$.\end{proof}

\subsection{Random infinite looptrees.}\label{sec:RandomInfiniteLooptrees} We now define infinite planar maps out of the infinite trees $\Tinfwb=\Tinfwb(\rhow,\rhob)$. 

\smallskip

\noindent\textit{The critical case.} When $(\rhow,\rhob)$ is critical, $\Tinfwb$ is a.s.\ locally finite. We extend the mapping $\Loop$ to $\tr\in\Tloc$ by defining $\Loop(\tr)$ as the consistent sequence of maps $(\Loop(\B_{2R}(\tr)) : R\geq 0 )$. This mapping is continuous on $\Tloc$ for the local topology. When $\tr$ is infinite and one-ended, $\Loop(\tr)$ is an infinite looptree, that is, an edge-outerplanar map whose root face is the unique infinite face. Then, the random infinite looptree $\Linf=\Linf(\rhow,\rhob)$ is defined by \[\Linf:=\Loop(\Tinfwb).\] See Figure \ref{fig:InfiniteLooptrees} for an illustration. Note that similar infinite looptrees also appear in \cite{bjornberg_random_2015}.

\smallskip

\noindent\textit{The subcritical case.} When $(\rhow,\rhob)$ is subcritical, $\Tinfwb$ has a unique vertex $u_\infty$ with infinite degree. Since $u_\infty$ has odd height, the sequence $(\B_r(\Loop(\Blr_R(\Tinfwb))) : R\geq 0)$ is eventually stationary, for every $r\geq 0$. Consequently, we define $\Linf=\Linf(\rhow,\rhob)$ as the local limit
\begin{equation}\label{eqn:DefinitionLinfSubcritical}
	\Linf:=\lim_{R\rightarrow \infty}\Loop(\Blr_R(\Tinfwb)).
\end{equation} Although $\Linf$ is not a looptree in the aforementioned sense, we keep the notation $\Linf=\Loop(\Tinfwb)$. The map $\Linf$ can also be obtained by gluing onto vertices $i\in \Z$ the independent looptrees $\Lt_i:=\Loop(\Tr_i)$, with $(\Tr_i:i\in \Z)$ as in Lemma \ref{lem:LawSizeBiasedGW} (see Figure \ref{fig:InfiniteLooptrees} for an illustration). From the above arguments, we get the following result.

\begin{figure}[ht]
	\centering
	\includegraphics[scale=.88]{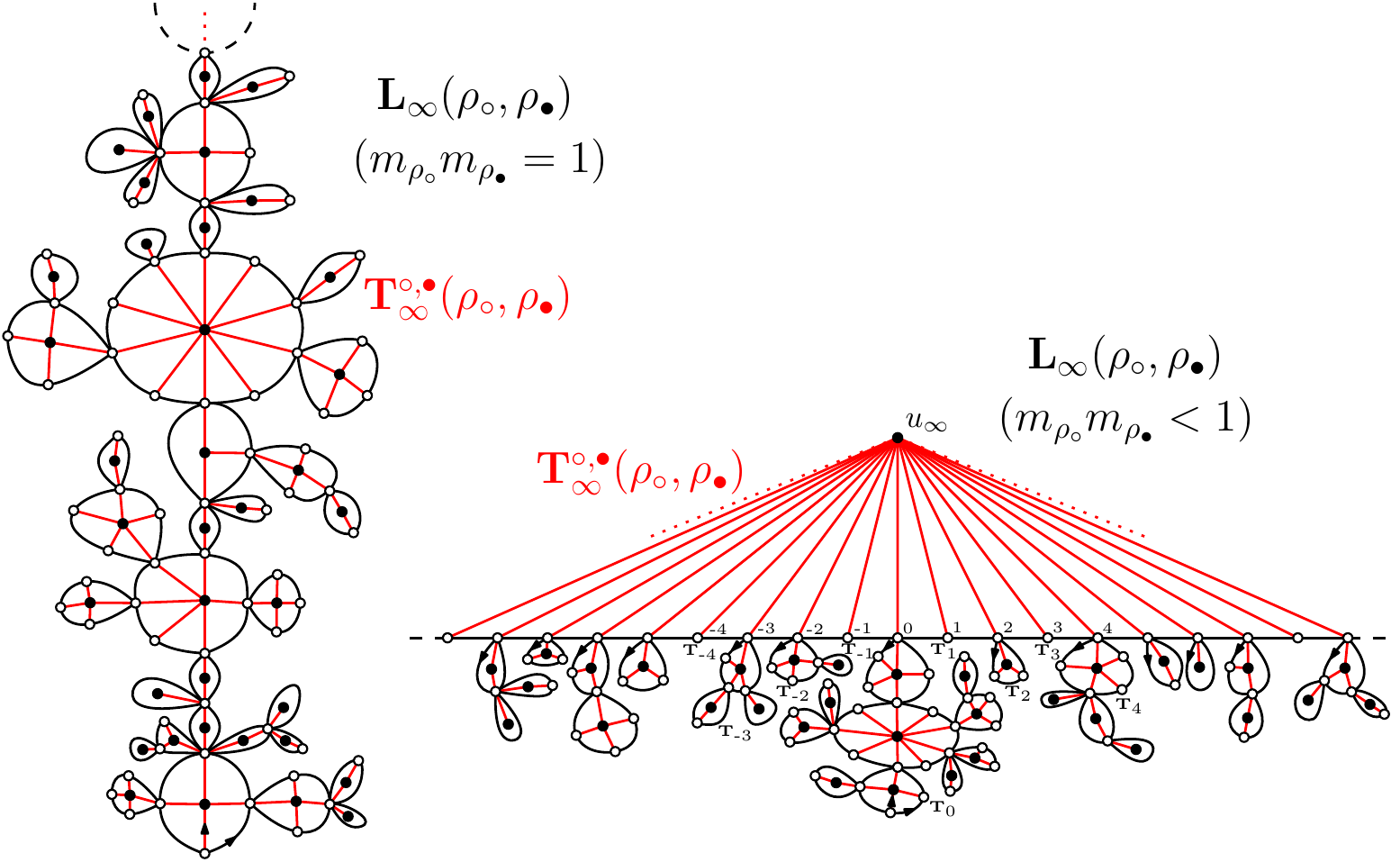}
	\caption{The infinite planar map $\Linf$ and the associated tree $\Tinfwb$.}
	\label{fig:InfiniteLooptrees}
	\end{figure}
	
\smallskip

\begin{Lem}\label{lem:ConvergenceLocalLooptrees} Let $(\rhow,\rhob)$ be a (sub)critical pair of offspring distributions such that $\rhow$ is geometric and $\rhob$ has no exponential moment. For every $k\geq 1$, let $\Twb_k$ be a tree with distribution $\GWrr^{(k)}$. Then, in the local sense,
	\[L_k:=\Loop(\Twb_k) \underset{k\rightarrow \infty}{\overset{(d)}{\longrightarrow}} \Linf(\rhow,\rhob).\]
\end{Lem}

Note that the internal faces of $\Linf=\Loop(\Tinfwb)$ are all finite in the critical case, while there is a unique infinite internal face in the subcritical case. 

\subsection{Local limits of Boltzmann maps with a boundary}\label{sec:LocalLimitsSubsection}

We now deal with the $\q$-$\IBHPM$ introduced in \eqref{eqn:LocalLimitBoltzmannMap}. This map, denoted by $\Minf=\Minf(\q)$, is a.s.\ one-ended with an infinite boundary (see \cite[Theorem 7]{curien_peeling_2016}).

The definitions of the scooped-out map and the irreducible components extend to any $\m\in\M_\infty$. We are now interested in the continuity of $\Scoop$ for the local topology.

\begin{Lem}\label{lem:ConvergenceBoundary} Let $(\m_k : k\in \N)$ be a sequence of  maps in $\M$, and $\m_\infty$ a one-ended infinite map such that $\m_k \rightarrow \m_\infty$ in the local sense, as $k\rightarrow \infty$. Then, in the local sense,
\[  \Scoop(\m_k) \underset{k\rightarrow \infty}{\longrightarrow}  \Scoop(\m_\infty) .\]
\end{Lem}
\begin{proof}First, if $(\#\Bm_k : k\geq 1)$ is bounded, there exists $R\geq 0$ such that for every $k\geq 1$, $\Bm_k\subseteq \B_R(\m_k)$ and the result follows. Thus, we can assume that $\#\Bm_k \rightarrow \infty$ as $k\rightarrow \infty$. 

For every $k\in \N\cup\{\infty\}$, let $p(k):=\#\Bm_k/2$ and denote by $(v_k(0),v_k(1),\ldots,v_k(p(k)))$ the vertices of the root face of $\m_k$ from the origin, in right contour order (with repetition). We use the notation $(v_k(0),v_k(-1),\ldots,v_k(-p(k)))$ for the left contour order. 

Let $r\geq 0$. We now prove that there exists $R\geq 0$ and $K\geq 1$ such that for every $k\geq K$,
\begin{equation}\label{eqn:RangeBoundary}
	\mathrm{V}(\B_r(\m_k))\cap\{v_{k}(l) : \vert l \vert >R\}=\emptyset.
\end{equation} We proceed by contradiction. By local convergence, the sequence $(\#\mathrm{V}(\B_r(\m_k)): k\geq 0)$ is bounded. Moreover, for every $v\in \mathrm{V}(\B_r(\m_k))$ we have 
\[\#\{ -p(k) \leq l \leq p(k) : v_{k}(l)=v \}\leq \deg_{\m_k}(v)\leq \sup_{u\in \mathrm{V}(B_r(\m_k))}{\deg_{\m_k}(u)},\] which is also bounded. Therefore, there exists $M\geq 0$ such that for every $k\geq 0$,
\[\#\{-p(k)\leq l \leq p(k) : v_k(l)\in \mathrm{V}(\B_r(\m_k))\}\leq M.\] Let $N\geq 0$. By assumption, there exists infinitely many $k$ such that $p(k) > 2M(N+2)$ and \[\mathrm{V}(\B_r(\m_k))\cap\{v_{\m_k}(l) : \vert l \vert >M(N+2)\}\neq\emptyset.\] As a consequence, in the cycle $(-p(k),\ldots,p(k))$, there exists two distinct sequences of consecutive indices $(i,\ldots,i+x)$ and $(j,\ldots,j+y)$ such that $x,y\geq N+2$ and
\[\mathrm{V}(\B_r(\m_k))\cap\{v_{k}(l) : i\leq l \leq i+x \}=\{v_{k}(i),v_{k}(i+x)\},\] and similarly for $(j,\ldots,j+y)$. In particular, the sets $E_1:=\{v_{k}(i+1),\ldots,v_{k}(i+x-1)\}$ and $E_2:=\{v_{k}(j+1),\ldots,v_{k}(j+y-1)\}$ are disjoint. Indeed, a vertex $v\in E_1\cap E_2$ would disconnect $\Scoop(\m_k)$ in two submaps each containing a vertex at distance less than $r$ from the origin, which is in contradiction with $v\notin \B_r(\m_k)$. Now, for every $-p(k)\leq i < p(k)$, $(v_k(i),v_k(i+1))$ is an edge of $\Scoop(\m_k)$. Therefore, $\{((v_{k}(l),v_{k}(l+1)) : i<l\leq i+N+1\}$ and $\{((v_{k}(l),v_{k}(l+1)) : j<l\leq j+N+1\}$ are disjoint sets of $N$ half-edges contained in $\B_{r+N}(\m_k)\backslash \B_r(\m_k)$. This holds for infinitely many $k\geq 1$, thus for $\m_\infty$. Since $\m_\infty$ has one end and $N$ is arbitrary, this is a contradiction.

Let us choose $R$ and $K$ such that assertion \eqref{eqn:RangeBoundary} holds for every $k\geq K$ (and thus for $\m_\infty$). For every $k\geq K$, let $\langle  v_{k}(-R),\ldots,v_{k}(R)\rangle$ be the sub-map induced by the $R$ first half-edges of $\Scoop(\m_k)$ in left and right contour order. We denote by $H$ the measurable function such that $\langle  v_{k}(-R),\ldots,v_{k}(R)\rangle=H(\m_k)=H(\B_R(\m_k))$. By \eqref{eqn:RangeBoundary}, we have for every $k\geq K$
\[\B_r(\Scoop(\m_k))=\B_r(H(\B_R(\m_k))\underset{k\rightarrow \infty}{\longrightarrow}\B_r(H(\B_R(\m_\infty))=\B_r(\Scoop(\m_\infty)),\] which concludes the proof.\end{proof}

Recall that when $\m\in \M_\infty$ has a unique infinite irreducible component, it is called the \textit{core} of $\m$, and denoted by $\Core(\m)$. We are now ready to prove Theorem \ref{th:LocalLimits}. 

\begin{proof}[Proof of Theorem \ref{th:LocalLimits}] For every $k\geq 1$, let $M_k$ be a map with law $\Pqk$. By Corollary \ref{cor:LawConditionedTreeComponents}, $\Twb_k:=\Treeb(M_k)$ has law $\GWnn^{(2k+1)}$. By \eqref{eqn:LocalLimitBoltzmannMap} and Lemmas \ref{lem:ConvergenceBoundary} and \ref{lem:ConvergenceLocalLooptrees}, we have  \[\Scoop(M_k) \underset{k\rightarrow \infty}{\overset{(d)}{\longrightarrow}} \Scoop(\Minf) \quad \text{and} \quad \Scoop(M_k)=\Loop(\Twb_k) \underset{k\rightarrow \infty}{\overset{(d)}{\longrightarrow}} \Linf(\nuw,\nub).\] Lemma \ref{lem:CriticityJSTree} and Proposition \ref{prop:PropertyNu} conclude the first part of the proof. For $a\in[3/2,2)$, $\Scoop(\Minf)$ has only finite internal faces, which are the boundaries of the irreducible components of $\Minf$. Since $\Scoop(\Minf)$ and $\Minf$ are one-ended, these components are necessarily finite. For $a\in(2,5/2]$, $\Scoop(\Minf)$ has a unique infinite internal face, which is the boundary of an infinite irreducible component. Since $\Minf$ is one-ended, the other components are finite, and $\Minf$ has a well-defined core. Moreover, $\Core(\Minf)$ is one-ended with an infinite simple boundary, and thus homeomorphic to the half-plane.\end{proof}

\medskip

\noindent\textbf{Local limits: the subcritical and dense regimes.} When $\q$ is of type $a\in[3/2,2)$, $\Minf$ can be built from the looptree $\Linf(\nuw,\nub)$ and a collection of independent Boltzmann maps. This generalizes \cite[Theorem 4]{baur_uniform_2016} which deals with subcritical quadrangulations. 

We define a fill-in mapping that associates to a one-ended tree $\tr\in\Tloc$ and a collection $(\ms_u : u\in \tb)$ of finite maps with a simple boundary of respective perimeter $\deg(u)$ the map 
\[\TC^{-1}\left(\tr,\left( \ms_u : u\in \tb \right)\right),\] obtained from $\lt:=\Loop(\tr)$ by gluing the map $\ms_u$ in the face of $\lt$ associated to $u$, for every $u\in\tb$. We keep the notation $\TC^{-1}$ by consistency, although we consider infinite trees. This mapping is continuous with respect to the natural topology.

\begin{Prop}\label{prop:LocalLimitSubcriticalDense} Let $\q$ be of type $a\in[3/2,2)$, and $\Tinfwb=\Tinfwb(\nuw,\nub)$. Conditionally on $\Tinfwb$, let $(\widehat{M}_u : u\in (\Tinfwb)_\bullet)$ be independent maps with a simple boundary and law $\widehat{\mathbb{P}}^{(\deg(u)/2)}_{\q}$. Then, the map \[M_\infty=\TC^{-1}\left(\Tinfwb,\left( \widehat{M}_u : u\in (\Tinfwb)_\bullet\right)\right)\] has the law of the $\q$-$\IBHPM$.
\end{Prop}

\begin{proof}The proof closely follows that of \cite[Theorem 4]{baur_uniform_2016}. For every $\tr \in \Tloc$ and every $R\geq 1$, let $\Cut_R(\tr)$ be the subtree of $\tr$ made of vertices $u\in\tr$ such that $\vert u \vert \leq 2R$. Consistently, if $\m=\TC^{-1}(\tr,(\ms_u : u\in \tb))$, $\Cut_R(\m)$ is the map $\TC^{-1}(\Cut_R(\tr),(\ms_u : u\in \Cut_R(\tr)_\bullet))$.

Let $R\geq 1$ and for every $k\geq 0$, let $M_k$ be a map with law $\Pqk$. Let $\m\in\M$ and $(\tr,(\ms_u : u\in \tb))=\TC(\m)$. By Proposition \ref{prop:LawTreeComponents} and \ref{prop:ConvergenceKestenTree}, we have \begin{align*}
	\Pqk(\Cut_R(M)=\m)=& \ \GWnn^{(2k+1)}(\Cut_R(T)=\tr)\prod_{u\in \tb}{\widehat{\mathbb{P}}^{(\deg(u)/2)}_\q}(\ms_u)\\
	\underset{k\rightarrow \infty}{\longrightarrow}& \ P(\Cut_R(\Tinfwb)=\tr)\prod_{u\in \tb}{\widehat{\mathbb{P}}^{(\deg(u)/2)}_\q}(\ms_u)=P(\Cut_R(M_\infty)=\m).
\end{align*} This concludes since $\B_R(\m)=\B_R(\Cut_R(\m))$ if $\m=\TC^{-1}(\tr,(\ms_u : u\in \tb))$ with $\tr\in\Tloc$.\end{proof}

\medskip

\noindent\textbf{Local limits: the dilute and generic regimes.} When $\q$ is of type $a\in(2,5/2]$, $\Minf$ cannot be fully described using finite maps. We believe that the finite irreducible components of $\Minf$ are independent Boltzmann maps with a simple boundary (conditionally on $\Scoop(\Minf)$). Moreover, we conjecture that there exists a distribution $\widehat{\mathbb{P}}_\q^{(\infty)}$ on one-ended maps with an infinite simple boundary such that $\Pqks \Rightarrow \widehat{\mathbb{P}}_\q^{(\infty)}$ as $k\rightarrow \infty$, and that $\Core(\Minf)$ has law~$\widehat{\mathbb{P}}_\q^{(\infty)}$. This result is proved for quadrangulations in \cite[Proposition 6]{curien_uniform_2015}, but relies on enumeration results for quadrangulations with a simple boundary that are unknown for general maps.

\section{The non-generic critical case with parameter $\alpha=3/2$}\label{sec:TheA2Case}

We now deal with the parameter $\alpha=3/2$ ($a=2$) that has been excluded so far. The results of Section \ref{sec:BoltzmannLaws} still hold by considering $a=2$ as part of the dense regime if $F'(\rq)=\infty$, and of the dilute regime if $F'(\rq)<\infty$. For instance, the proofs of Propositions \ref{prop:SingularExpansionF} and \ref{prop:SingularExpansionFs} can be slightly adapted. For the former, we use \cite[Proposition 1.5.9 a-b]{bingham_regular_1989} to check that the assumption \cite[Equation (8.1.11 a-c)]{bingham_regular_1989} of Karamata's theorem~\cite[Theorem 8.1.6]{bingham_regular_1989} is satisfied. For the latter, when $F'(\rq)=\infty$, we use the so-called de Bruijn conjugate of a slowly varying function \cite[Theorem 1.5.13]{bingham_regular_1989} instead of Lemma \ref{lem:RegularlyVaryingInversion}. The issue comes from Proposition \ref{prop:PropertyNu}, because in this case Karamata's theorem merely provides information on the tail of the size-biased version of $\nu$, see \cite[Equation (8.1.11 a-c)]{bingham_regular_1989}. We now bypass this difficulty by using a special weight sequence introduced in \cite{ambjorn_generalized_2016}, and by calling on de Haan theory \cite[Chapter 3]{bingham_regular_1989}. Let us start with a general statement regarding the criticality of the tree of components that is a consequence of \eqref{eqn:AsymptoticsFk} and \eqref{eqn:MeanNu}.

\begin{Lem}
	Let $\q$ be a weight sequence of type $a=2$. Then, $\nu$ is critical if and only if 
	\[\sum_{k=1}^\infty \frac{\ell(k)}{k}=\infty.\] 
\end{Lem}

The special weight sequence $\q^*=(q^*_k : k\in \N)$ introduced in \cite{ambjorn_generalized_2016} (see also \cite[Section 5]{budd_geometry_2017}) is defined by \begin{equation}\label{eqn:Qstar}
	q^*_k:=\frac{1}{4 }6^{1-k} \frac{\Gamma(k-3/2)}{\Gamma(k+5/2)}\mathbf{1}_{k\geq 2} \quad k\in \N.
\end{equation} The sequence $\q^*$ is admissible, critical, and of type $a=2$. We will prove the following. 

\begin{Prop}\label{prop:SpecialCaseA2}
	Let $\q^*$ be the sequence defined by \eqref{eqn:Qstar}. Then, $(\nuw,\nub)$ is critical, 
\[\nu([k,\infty))\underset{k\rightarrow \infty}{\sim} \frac{1}{k\log^2(k)}, \quad \text{and} \quad \nub([k,\infty))\underset{k\rightarrow \infty}{\sim} \frac{3}{k\log^2(k)}.\] In particular, $\nu$ and $\nub$ are in the domain of attraction of a Cauchy distribution.
\end{Prop}
\begin{Rk}When $a=2$, $(\nuw,\nub)$ can either be subcritical or critical. In the critical case (which includes the standard setting of \cite{le_gall_scaling_2011,borot_recursive_2012} and the sequence $\q^*$), the results of Theorem \ref{th:LocalLimits} and Proposition \ref{prop:LocalLimitSubcriticalDense} hold. However, Proposition \ref{prop:SpecialCaseA2} suggests that $\nub$ has a very heavy tail, meaning that $\partial\Minf$ has very large loops. In the subcritical case, Theorem \ref{th:LocalLimits} also holds provided that $\nub$ has no exponential moment. This can be proved by using the analogue of Proposition \ref{prop:SingularExpansionFs} to ensure that $\rsq=\rq F^2(\rq)$ (in this case, the slowly varying correction vanishes at infinity and the expansion is singular). Finally, we expect the scaling limit of the boundary to be a circle, but the normalizing sequence to be negligible compared to the perimeter $2k$ of the map (typically of order $k/\log(k)$).	
\end{Rk}

\medskip

\noindent\textbf{Enumerative results.} The sequence $\q^*$ is convenient because we have an explicit formula for the partition functions $(F_k : k\geq 0)$ by \cite[Lemma 14]{budd_geometry_2017} and \cite[Equation (7)]{budd_geometry_2017}:
\[F_k=\frac{3}{4} \frac{6^{k}}{(k+3/2)(k+1/2)}, \quad \text{and} \quad F(x)=\frac{1}{4x}-\frac{3}{4(6x)^{3/2}}(1-6x)\log\left( \frac{1+\sqrt{6x}}{1-\sqrt{6x}} \right).\] Consequently, $\rq=1/6$ and we deduce the asymptotic expansions as $x \rightarrow \rq^-$
\begin{align}\label{eqn:A2ExpansionF}
	F(x)&=\frac{3}{2}+\frac{3}{4}\left(1-\frac{x}{\rq}\right)\log\left(1-\frac{x}{\rq}\right)+\frac{3}{2}\left(1-\log(2)\right)\left(1-\frac{x}{\rq}\right)(1+o(1)),\\
	F'(x)&=-\frac{9}{2}(3-2\log(2))-\frac{9}{2}\log\left(1-\frac{x}{\rq}\right)+o(1).\label{eqn:A2ExpansionDF}
\end{align} Unlike the previous cases, an expansion of $\Fs$ is not sufficient; we rather need an expansion of its derivative. The function $P(x)=xF^2(x)$ is again continuous increasing with inverse $P^{-1}$. Moreover, we have as $x\rightarrow \rq^-$
\begin{equation}\label{eqn:A2ExpansionP}
	P(x)=P(\rq)+P(\rq)\left(1-\frac{x}{\rq}\right)\log\left(1-\frac{x}{\rq}\right)+P(\rq)(2\log(2)-1)\left(1-\frac{x}{\rq}\right)(1+o(1)).
\end{equation} We put $c^*:=2\log(2)-1$. Let us define the function $R$ and its inverse $R^{-1}$ both on $[0,1]$ by
\begin{equation}\label{eqn:A2InverseR}
	R(x):=\frac{1}{P(\rq)}\left(P(\rq)-P(\rq(1-x)) \right) \quad \text{and} \quad R^{-1}(y)=1-\frac{1}{\rq}P^{-1}\left(P(\rq)(1-y)\right).
\end{equation} The expansion of $R$ reads $R(x)=-x\log(x)-c^*x+o(x)$, as $x\rightarrow 1^-$. We now need the Lambert $W$ function, which is the multivalued inverse of $x\mapsto x e^x$. We use the lower branch $\Wm$, continuous decreasing from $[-1/e,0)$ onto $(-\infty,-1]$, which satisfies \begin{equation}\label{eqn:LambertW}
	\Wm(-x)=\log\left( \frac{-x}{\Wm(-x)} \right) \quad \text{and} \quad \Wm(x\log(x))=\log(x), \quad x\in (0,1/e]. 
\end{equation} We also have $\Wm(-x)=\log(x)-\log(-\log(x))+o(1)$ as $x\rightarrow 0^+$. We introduce the function
\[Q(x):=R\left( \frac{-x}{\Wm(-x)} \right), \quad x\in(0,1/e],\] which is continuous increasing. By \eqref{eqn:LambertW}, its inverse function $Q^{-1}$ satisfies 
\begin{equation}\label{eqn:A2FormulaInverseQ}
	Q^{-1}(y)=-R^{-1}(y)\log\left(R^{-1}(y)\right) \quad \text{and} \quad R^{-1}(y)=\frac{-Q^{-1}(y)}{\Wm(-Q^{-1}(y))}, \quad y\in(0,R(1/e)].
\end{equation} Using the above expansions, we get as $x,y \rightarrow 0^+$
\begin{equation}\label{eqn:A2ExpansionQ}
	Q(x)=x-c^*\frac{x}{\log(x)}+o\left(\frac{x}{\log(x)}\right) \quad \text{and} \quad Q^{-1}(y)=y-c^*\frac{y}{\log(y)}+o\left(\frac{
	y}{\log(y)}\right).
\end{equation} Together with \eqref{eqn:A2FormulaInverseQ} and the expansion of $W_{-1}$, this yields 
\begin{equation}\label{eqn:A2ExpansionRInverse}
	R^{-1}(y)=-\frac{y}{\log(y)}-\frac{y\log(-\log(y))}{\log^2(y)}-c^*\frac{y}{\log^2(y)}+o\left(\frac{y}{\log^2(y)}\right) \quad \text{as } y\rightarrow 0^+.
\end{equation} Finally, by \eqref{eqn:A2InverseR} we obtain
\begin{multline}\label{eqn:A2ExpansionPInverse}
	P^{-1}(y)=\rq +\rq\left(1-\frac{y}{P(\rq)}\right)\frac{1}{\log\left(1-\frac{y}{P(\rq)}\right)}+\rq\left(1-\frac{y}{P(\rq)}\right)\frac{\log\left(-\log\left(1-\frac{y}{P(\rq)}\right)\right)}{\log^2\left(1-\frac{y}{P(\rq)}\right)}\\
	+\rq c^* \left(1-\frac{y}{P(\rq)}\right)\frac{1}{\log^2\left(1-\frac{y}{P(\rq)}\right)}(1+o(1)) \quad \text{as } y\rightarrow P(\rq)^-.
\end{multline} This proves that $\rsq=P(\rq)$. By differentiating both sides in the equation of Lemma \ref{lem:RelationFFs} and using \eqref{eqn:A2ExpansionF}, \eqref{eqn:A2ExpansionDF} and \eqref{eqn:A2ExpansionPInverse} we obtain the wanted expansion of $\Fs'$: as $y\rightarrow P(\rq)^-$,
\begin{equation}\label{eqn:A2ExpansionFsD}
	\Fs'(y)=2+ \frac{2}{\log\left(1-\frac{y}{P(\rq)}\right)}+\frac{2\log\left(-\log\left(1-\frac{y}{P(\rq)}\right)\right)}{\log^2\left(1-\frac{y}{P(\rq)}\right)}-\frac{2(3-2\log(2))}{\log^2\left(1-\frac{y}{P(\rq)}\right)}(1+o(1)).
\end{equation}

\medskip

\noindent\textbf{The tree of components.} We are now interested in properties of the tails of $\nu$ and $\nub$. To do so, we need estimates on the derivative of $L_\nu$. Recalling \eqref{eqn:GFnu} and \eqref{eqn:A2ExpansionFsD}, we obtain
\begin{equation}\label{eqn:A2ExpansionLnuD}
	-L_\nu'(t)=1+ \frac{1}{\log(2t)}+\frac{\log\left(-\log\left(2t\right)\right)}{\log^2\left(2t\right)}-\frac{3-2\log(2)}{\log^2\left(2t\right)}+o\left(\frac{1}{\log^2\left(t\right)}\right), \quad \text{as } t\rightarrow 0^+.
\end{equation} Since $\nu$ is critical, the Laplace transform $L_{\bar{\nu}}$ of the size-biased measure $\bar{\nu}$ equals $-L_\nu'$. Thus, 
\begin{equation}\label{eqn:A2LimitLaplaceNuBar}
	\frac{L_{\bar{\nu}}\left( \frac{1}{\lambda x} \right)-L_{\bar{\nu}}\left( \frac{1}{ x} \right)}{\log^2(x)} \underset{x \rightarrow \infty}{\longrightarrow}\log(\lambda), \quad \forall \ \lambda>0.
\end{equation} Let us introduce a notation for the tail of $\bar{\nu}$, say $T(x):=\sum_{k\geq x}{k\nu(k)}$ for $x\in \R$. By de Haan's Tauberian theorem \cite[Theorem 3.9.1]{bingham_regular_1989}, \eqref{eqn:A2LimitLaplaceNuBar} is equivalent to 
\begin{equation}\label{eqn:A2LimitT}
	\frac{T(\lambda x) -T(x)}{\log^2(x)} \underset{x \rightarrow \infty}{\longrightarrow}\log(\lambda), \quad \forall \ \lambda>0.
\end{equation} The function $T$ is said to be in the class $\Pi_{\log^2}$ with index $1$. By an integration by parts,
\begin{equation}\label{eqn:A2IPP}
	x\nu((x,\infty))=T(x)-x\int_{x}^\infty{\frac{T(t)}{t^2}\textup{d}t}, \quad x>0.
\end{equation} This finally proves Proposition \ref{prop:SpecialCaseA2} by de Haan's Theorem \cite[Theorem 3.7.3]{bingham_regular_1989}, \eqref{eqn:A2LimitT} and \eqref{eqn:A2IPP}.

\begin{Ack}Many thanks to Grégory Miermont for enlightening discussions and for attentively reading this work. I would also like to thank warmly Erich Baur, Jérémie Bouttier, Timothy Budd, Nicolas Curien and Igor Kortchemski for very useful discussions and comments. Finally, I am deeply indebted to the anonymous referees for their careful rereading, crucial corrections and remarks.
\end{Ack}

\end{document}